\documentclass[11pt,letterpaper]{article}

\usepackage{graphicx,subfigure}

 \newcommand{\Lip}{\mathrm{Lip}}         \newcommand{\edge}{M^{\mathrm{gr}}} \newcommand{\goodtimes}{\vvmathbb{T}}  \newcommand{\tmix}{t_{\mathrm{mix}}}

\newcommand{\newgraph}{\vvmathbb{G}} \newcommand{\hit}{t_{\mathrm{hit}}} \newcommand{\level}{\lambda}
\newcommand{\nbors}{\tilde G}

\def\stocmode{0} \def\jamesmode{0} \def\arxivmode{1} \def\fastmode{0}   \def\showauthornotes{1}  \def\showkeys{0} \def\showdraftbox{1} \def\showcolorlinks{1} \def\usemicrotype{1} \def\showfixme{1} 

\ifnum\fastmode=0
\usepackage{etex}
\fi

\ifnum\fastmode=0
\usepackage[l2tabu, orthodox]{nag}
\fi

\usepackage{xspace,enumerate}

\usepackage[dvipsnames]{xcolor}

\ifnum\fastmode=0
\usepackage[T1]{fontenc}
\usepackage[full]{textcomp}
\fi

\usepackage[american]{babel}
\usepackage{mathtools}

\usepackage{amsthm}

\newtheorem{theorem}{Theorem}[section]
\newtheorem*{theorem*}{Theorem}

\newtheorem*{proposition*}{Proposition}
\newtheorem{lemma}[theorem]{Lemma}
\newtheorem*{lemma*}{Lemma}
\newtheorem{corollary}[theorem]{Corollary}
\newtheorem*{conjecture*}{Conjecture}
\newtheorem{fact}[theorem]{Fact}
\newtheorem*{fact*}{Fact}

\newtheorem*{exercise*}{Exercise}

\newtheorem*{hypothesis*}{Hypothesis}

\theoremstyle{definition}
\newtheorem{definition}[theorem]{Definition}

\newtheorem{exercise-easy}[theorem]{Exercise}
\newtheorem{exercise-med}[theorem]{Exercise}
\newtheorem{exercise-hard}[theorem]{Exercise$^\star$}

\newtheorem*{claim*}{Claim}

\newtheorem{remark}[theorem]{Remark}
\newtheorem*{remark*}{Remark}

\newtheorem*{observation*}{Observation}

\usepackage[
letterpaper,
top=0.7in,
bottom=0.9in,
left=1in,
right=1in]{geometry}

\ifnum\arxivmode=0
\usepackage{newpxtext} % T1, lining figures in math, osf in text
\usepackage{textcomp} % required for special glyphs
\usepackage[varg,bigdelims]{newpxmath}
\usepackage[scr=rsfso]{mathalfa}% \mathscr is fancier than \mathcal
\usepackage{bm} % load after all math to give access to bold math
\let\mathbb\varmathbb
\fi

\ifnum\arxivmode=1
\usepackage[varg]{pxfonts} % varg - uses nicer g,v,y,w
\usepackage{textcomp} % required for special glyphs
\usepackage[scr=rsfso]{mathalfa}% \mathscr is fancier than \mathcal
\usepackage{bm} % load after all math to give access to bold math
\let\vvmathbb\mathbb
\fi

\ifnum\showkeys=1
\usepackage[color]{showkeys}
\fi

\ifnum\arxivmode=1
\makeatletter
\@ifpackageloaded{hyperref}
{}
{\usepackage{hyperref}
\PassOptionsToPackage{pagebackref=true}{hyperref}
}
\else
\makeatletter
\@ifpackageloaded{hyperref}
{}
{\usepackage[pagebackref]{hyperref}}
\fi

\ifnum\showcolorlinks=1
\hypersetup{
pagebackref=true,
colorlinks=true,
urlcolor=blue,
linkcolor=blue,
citecolor=OliveGreen,
linktocpage=false}
\fi

\ifnum\showcolorlinks=0
\hypersetup{
colorlinks=false,
pdfborder={0 0 0}}
\fi

\usepackage{prettyref}

\newcommand{\savehyperref}[2]{\texorpdfstring{\hyperref[#1]{#2}}{#2}}

\newrefformat{eq}{\savehyperref{#1}{\textup{(\ref*{#1})}}}
\newrefformat{lem}{\savehyperref{#1}{Lemma~\ref*{#1}}}
\newrefformat{def}{\savehyperref{#1}{Definition~\ref*{#1}}}
\newrefformat{thm}{\savehyperref{#1}{Theorem~\ref*{#1}}}
\newrefformat{cor}{\savehyperref{#1}{Corollary~\ref*{#1}}}
\newrefformat{cha}{\savehyperref{#1}{Chapter~\ref*{#1}}}
\newrefformat{sec}{\savehyperref{#1}{Section~\ref*{#1}}}
\newrefformat{app}{\savehyperref{#1}{Appendix~\ref*{#1}}}
\newrefformat{tab}{\savehyperref{#1}{Table~\ref*{#1}}}
\newrefformat{fig}{\savehyperref{#1}{Figure~\ref*{#1}}}
\newrefformat{hyp}{\savehyperref{#1}{Hypothesis~\ref*{#1}}}
\newrefformat{alg}{\savehyperref{#1}{Algorithm~\ref*{#1}}}
\newrefformat{rem}{\savehyperref{#1}{Remark~\ref*{#1}}}
\newrefformat{item}{\savehyperref{#1}{Item~\ref*{#1}}}
\newrefformat{step}{\savehyperref{#1}{step~\ref*{#1}}}
\newrefformat{conj}{\savehyperref{#1}{Conjecture~\ref*{#1}}}
\newrefformat{fact}{\savehyperref{#1}{Fact~\ref*{#1}}}
\newrefformat{prop}{\savehyperref{#1}{Proposition~\ref*{#1}}}
\newrefformat{prob}{\savehyperref{#1}{Problem~\ref*{#1}}}
\newrefformat{claim}{\savehyperref{#1}{Claim~\ref*{#1}}}
\newrefformat{relax}{\savehyperref{#1}{Relaxation~\ref*{#1}}}
\newrefformat{red}{\savehyperref{#1}{Reduction~\ref*{#1}}}
\newrefformat{part}{\savehyperref{#1}{Part~\ref*{#1}}}
\newrefformat{ex}{\savehyperref{#1}{Exercise~\ref*{#1}}}
\newrefformat{ques}{\savehyperref{#1}{Question~\ref*{#1}}}

\newcommand{\Sref}[1]{\hyperref[#1]{\S\ref*{#1}}}

\usepackage{nicefrac}
\ifnum\fastmode=0
\ifnum\usemicrotype=1
\usepackage{microtype}
\fi
\fi

\ifnum\showauthornotes=1
\newcommand{\Authornote}[2]{{\sffamily\small\color{blue}{[#1: #2]}}\medskip}
\newcommand{\Authornotecolored}[3]{{\sffamily\small\color{#1}{[#2: #3]}}}
\newcommand{\Authorcomment}[2]{{\sffamily\small\color{gray}{[#1: #2]}}}
\newcommand{\Authorstartcomment}[1]{\sffamily\small\color{gray}[#1: }

\newcommand{\Authorfnote}[2]{\footnote{\color{red}{#1: #2}}}
\newcommand{\Authorfixme}[1]{\Authornote{#1}{\textbf{??}}}
\newcommand{\Authormarginmark}[1]{\marginpar{\textcolor{red}{\fbox{\Large #1:!}}}}
\else
\newcommand{\Authornote}[2]{}
\newcommand{\Authornotecolored}[3]{}
\newcommand{\Authorcomment}[2]{}
\newcommand{\Authorstartcomment}[1]{}

\newcommand{\Authorfnote}[2]{}
\newcommand{\Authorfixme}[1]{}
\newcommand{\Authormarginmark}[1]{}
\fi

\ifnum\showfixme=0

\fi

\usepackage{boxedminipage}

\newcommand{\Esymb}{\mathbb{E}}
\newcommand{\Psymb}{\mathbb{P}}
\newcommand{\Vsymb}{\mathbb{V}}

\DeclareMathOperator*{\E}{\Esymb}
\DeclareMathOperator*{\Var}{\Vsymb}
\DeclareMathOperator*{\ProbOp}{\Psymb}

\renewcommand{\Pr}{\ProbOp}

\newcommand{\textparen}[1]{\text{(#1)}}

\ifx\because\undefined
\newcommand{\because}[1]{\textparen{because #1}}
\else
\renewcommand{\because}[1]{\textparen{because #1}}
\fi

\newcommand{\defeq}{\stackrel{\mathrm{def}}=}

\newcommand{\seteq}{\mathrel{\mathop:}=}

\newcommand\bdot\bullet

\ifx\mathds\undefined % use double stroke fonts if available
\newcommand{\Ind}{\mathbb I}
\else
\newcommand{\Ind}{\mathds 1}
\fi

\DeclareMathOperator{\supp}{supp}

\newcommand{\Z}{\mathbb Z}
\newcommand{\N}{\mathbb N}
\newcommand{\R}{\mathbb R}

\newcommand{\cA}{\mathcal A}
\newcommand{\cB}{\mathcal B}

\newcommand{\cE}{\mathcal E}

\newcommand{\cG}{\mathcal G}
\newcommand{\cH}{\mathcal H}

\newcommand{\cP}{\mathcal P}

\newcommand{\cT}{\mathcal T}

\newcommand{\cY}{\mathcal Y}

\renewcommand{\leq}{\leqslant}
\renewcommand{\le}{\leqslant}
\renewcommand{\geq}{\geqslant}
\renewcommand{\ge}{\geqslant}

\ifnum\showdraftbox=1

\else

\fi

\let\epsilon=\varepsilon

\numberwithin{equation}{section}

\newcommand\MYcurrentlabel{xxx}

\newcommand{\MYstore}[2]{%
  \global\expandafter \def \csname MYMEMORY #1 \endcsname{#2}%
}

\newcommand{\MYload}[1]{%
  \csname MYMEMORY #1 \endcsname%
}

\newcommand{\MYnewlabel}[1]{%
  \renewcommand\MYcurrentlabel{#1}%
  \MYoldlabel{#1}%
}

\newcommand{\MYdummylabel}[1]{}

\newcommand{\torestate}[1]{%
  \let\MYoldlabel\label%
  \let\label\MYnewlabel%
  #1%
  \MYstore{\MYcurrentlabel}{#1}%
  \let\label\MYoldlabel%
}

\newcommand{\restatetheorem}[1]{%
  \let\MYoldlabel\label
  \let\label\MYdummylabel
  \begin{theorem*}[Restatement of \prettyref{#1}]
    \MYload{#1}
  \end{theorem*}
  \let\label\MYoldlabel
}

\newcommand{\restatelemma}[1]{%
  \let\MYoldlabel\label
  \let\label\MYdummylabel
  \begin{lemma*}[Restatement of \prettyref{#1}]
    \MYload{#1}
  \end{lemma*}
  \let\label\MYoldlabel
}

\newcommand{\restateprop}[1]{%
  \let\MYoldlabel\label
  \let\label\MYdummylabel
  \begin{proposition*}[Restatement of \prettyref{#1}]
    \MYload{#1}
  \end{proposition*}
  \let\label\MYoldlabel
}

\newcommand{\restatefact}[1]{%
  \let\MYoldlabel\label
  \let\label\MYdummylabel
  \begin{fact*}[Restatement of \prettyref{#1}]
    \MYload{#1}
  \end{fact*}
  \let\label\MYoldlabel
}

\newcommand{\restate}[1]{%
  \let\MYoldlabel\label
  \let\label\MYdummylabel
  \MYload{#1}
  \let\label\MYoldlabel
}

\newcommand{\addreferencesection}{
  \phantomsection
\ifnum\stocmode=0
  \addcontentsline{toc}{section}{References}
\else
  \addcontentsline{toc}{section}{References \hspace*{1in} --------- End of extended abstract ---------}
\fi

}

\newcommand{\e}{\epsilon}

\let\origparagraph\paragraph
\renewcommand{\paragraph}[1]{\vspace*{-15pt}\origparagraph{#1.}}

\allowdisplaybreaks

\usepackage{paralist}

\usepackage{comment}
 
\let\pref=\prettyref

\renewcommand{\Vsymb}{\mathrm{Var}}

\newcommand{\diam}{\mathrm{diam}}

\renewcommand{\Ind}{\bm 1}
\ifnum\jamesmode=0
\fi
\let\1\Ind

\newcommand\f{\varphi}

\begin{document}

\title{Diffusive estimates for random walks on \\ stationary random graphs of polynomial growth}

\author{Shirshendu Ganguly \and James R. Lee \and Yuval Peres}

\date{}

\maketitle

\linespread{1.00}% Give Palatino more leading (space between lines)

\begin{abstract} Let $(G,\rho)$ be a stationary random graph, and use $B^G_{\rho}(r)$ to denote the ball of radius $r$ about $\rho$ in $G$. Suppose that $(G,\rho)$ has annealed polynomial growth,
    in the sense that $\E[|B^G_{\rho}(r)|] \leq O(r^k)$ for some $k > 0$ and every $r \geq 1$.

Then there is an infinite sequence of times $\{t_n\}$ at which the random walk $\{X_t\}$ on $(G,\rho)$ is at most diffusive: Almost surely (over the choice of $(G,\rho)$), there is a number $C > 0$ such that \[ \E \left[\mathrm{dist}_G(X_0, X_{t_n})^2 \mid X_0 = \rho, (G,\rho)\right]\leq C t_n\qquad \forall n \geq 1\,. \] This result is new even in the case when $G$ is a stationary random subgraph of $\Z^d$. Combined with the work of Benjamini, Duminil-Copin, Kozma, and Yadin (2015), it implies that $G$ almost surely does not admit a non-constant harmonic function of sublinear growth.

To complement this, we argue that passing to a subsequence of times $\{t_n\}$ is necessary, as there are stationary random graphs of (almost sure) polynomial growth where the random walk is almost surely superdiffusive at an infinite subset of times. \end{abstract}

\tableofcontents

\section{Introduction}

It is a classical fact that the the standard random walk $\{X_n\}$ on $\Z^d$ exhibits
{\em diffusive behavior:}  $\E \|X_0-X_n\|_2^2 \asymp n$.
The well-known estimates of Varopoulos and Carne \cite{Carne85,Var85}
show that, if $\{X_n\}$ is random walk on a graph $G$ of polynomial growth,
then the speed can be at most slightly superdiffusive:  $\E d_G(X_0,X_n)^2 \leq O(n \log n)$,
where $d_G$ is the graph metric on $G$.

Kesten \cite{Kesten86} examined the distribution of the random walk on percolation clusters in $\Z^d$.
Suppose that $(G,\rho)$
is a stationary random subgraph of $\Z^d$.  This means that if $\{X_n\}$
is the random walk conditioned on $G$ with $X_0=\rho$, then
$(G,X_0) \stackrel{\mathrm{law}}{=} (G,X_1)$.
Kesten's argument can be used to show that, in this case,
\begin{equation}\label{eq:kesten}
\forall n \geq 1,\ \E \|X_0-X_n\|_2^2 \leq O(n) \qquad \textrm{almost surely over $(G,\rho)$.}
\end{equation}
On the other hand, Kesten's approach only works for the {\em extrinsic} Euclidean metric,
and not for the intrinsic metric $d_G$ (which can be arbitrarily larger).

Kesten asked whether \eqref{eq:kesten} holds for {\em any} (deterministic) subgraph $G$ of $\Z^d$.
Barlow and Perkins \cite{BP89} answered this negatively:  They exhibit a subgraph
of $\Z^2$ on which the Varopoulos-Carne bound is asymptotically tight (even for the Euclidean metric).

\medskip
\paragraph{Random walk and the growth of harmonic functions}

One motivation for studying situations in which Varopoulos-Carne can be improved
comes from the theory of harmonic functions and their role in geometric analysis
and in recent proofs of the central limit theorem for random graphs.
Indeed, this led the authors of \cite{benjamini} to study harmonic functions
in random environments.

Consider a random rooted graph $(G,\rho)$. We will assume that $G$ is locally finite and almost surely connected. Let $\{X_n\}$ denote the random walk conditioned on $(G,\rho)$.  Unless otherwise stated, we take $X_0=\rho$.

\begin{definition}
  $(G,\rho)$ is said to be {\em stationary} if $(G,X_0)\overset{\textrm{law}}{=}(G,X_1)$.
\end{definition}

Let $V_G$ be the vertex set of $G$, and let $d_G$ denote the graph metric on $G$.  For $x \in V_G$, we use the notation \[ B^G_{x}(r) = \left\{ y \in V_G : d_G(x,y) \leq r\right\}\,. \] The random graph $(G,\rho)$ has {\em annealed polynomial growth} if there exist constants $c,d > 0$ such that for $r \geq 1$, \begin{equation}\label{eq:strong-poly-growth} \E\left[|B^G_{\rho}(r)|\right] \leq c r^d\,. \end{equation} Say that the random walk $\{X_n\}$ is {\em at most diffusive} if there is a constant $C > 0$ such that \begin{equation}\label{eq:subdiffusive} \E\left[d_G(X_0, X_n)^2\right] \leq Cn \end{equation} for all $n \geq 1$.

We now state the main result of \cite{benjamini} for the special case of stationary random graphs. A harmonic function conditioned on $G$ is a map $h : V_G \to \R$ satisfying \[ \E[h(X_1) \mid X_0 = x] = h(x) \quad \forall x\in V_G\,. \] Say that $h$ has {\em sublinear growth} if for every infinite sequence $\{x_n\} \subseteq V_G$ with $d_G(\rho,x_n) \to \infty$, it holds that \[ \lim_{n \to \infty} \frac{|h(x_n)|}{d_G(\rho,x_n)} = 0\,. \]

\begin{theorem}[\cite{benjamini}] Suppose $(G,\rho)$ is a stationary random graph with annealed polynomial growth, and suppose the random walk on $(G,\rho)$ is at most diffusive in the sense of \eqref{eq:subdiffusive}. Then almost surely $G$ does not admit a non-constant harmonic function of sublinear growth. \end{theorem}

Our main result is that the diffusivity assumption can be removed.
Say that $(G,\rho)$ has {\em weakly annealed polynomial growth} if there are non-negative constants $c,s \geq 0$ such that for $r \geq 1$, \begin{equation}\label{eq:poly-growth} \E\left[\log |B^G_{\rho}(r)|\right] \leq s \log r + c\,. \end{equation} (Note that this is a {\em weaker} assumption than annealed polynomial growth.)

\begin{theorem}\label{thm:main} Suppose $(G,\rho)$ is a stationary random graph with weakly annealed polynomial growth. Then almost surely $G$ does not admit a non-constant harmonic function of sublinear growth. \end{theorem}

Our proof of \pref{thm:main} proceeds in the natural way:  We show that weakly annealed polynomial growth always yields a sequence of times at which the random walk is at most diffusive.

\begin{theorem}\label{thm:goodtimes-intro} If $(G,\rho)$ is a stationary random graph of annealed polynomial growth, then for every $\e > 0$, there is a constant $C > 0$ and an infinite (deterministic) sequence of times $\{t_n\}$ such that \[ \Pr\left({\E}\left[d_G(X_0, X_{t_n})^2 \mid (G,\rho)\right] \leq C t_n\right) \geq 1-\e\,. \] \end{theorem}

Note that \pref{thm:goodtimes-intro} is new even for stationary random subgraphs of $\Z^d$ since, in contrast
to Kesten's work, we are able to bound the speed of the random walk in the intrinsic metric.
To complement this result, we show that passing to a subsequence of times is necessary:
There are stationary random graphs of (almost sure) polynomial growth
on which the random walk is almost surely superdiffusive at an infinite subset of times.

\begin{theorem}[See \pref{thm:exceptional}]
\label{thm:badtimes} There is a stationary random graph $(G,\rho)$ of almost sure polynomial growth such that for an infinite (deterministic) sequence of times $\{t_n\}$, \[ \lim_{n \to \infty} \Pr\left(\E\left[d_G(X_0, X_{t_n})^2 \mid (G,\rho)\right] \geq  t_n (\log t_n)^{0.9} \right) = 1\,. \] \end{theorem}

We remark that instead of $(\log t_n)^{0.9}$, one could put $f(t_n)$ for any function satisfying $f(t) \leq o(\log t)$ as $t \to \infty$. This is almost tight as it nearly matches the Varopoulos-Carne estimate (see, e.g., \cite[Ch. 14]{Woess00}).
Our work leaves open the intriguing question of whether
whether \pref{thm:goodtimes-intro} holds for all times when $(G,\rho)$ is a stationary
random subgraph of $\Z^d$.

\subsection{The absence of non-constant sublinear growth harmonic functions}

Let us recall that the {\em entropy of $X_n$ conditioned on $(G,\rho)$:} 
\begin{equation}\label{entropy} H(X_n \mid (G,\rho)) = \sum_{x \in V_G} \Pr[X_n=x \mid (G,\rho)] \log \frac{1}{\Pr[X_n=x \mid (G,\rho)]}\,, \end{equation}
with the convention that $0 \log 0 = 0$.
Similarly we define $H\left((X_1,X_n) \mid (G,\rho)\right)$ to be the entropy of the joint distribution of $(X_1,X_n)$, conditioned on $(G,\rho)$. 
To simplify notation, we will denote $H(X_n \mid (G,\rho))$ and $H\left((X_1,X_n)\mid (G,\rho)\right)$ by $H_{(G,\rho)}(n)$ and $H_{(G,\rho)}(1,n)$, respectively.

Define the {\em annealed entropy} by \begin{equation}\label{annealed} H_n = \E[H_{(G,\rho)}(n)]\,. \end{equation}

Our proof of \pref{thm:main} is based on the main result of \cite{benjamini} which exploits connections between harmonic functions and the escape rate of random walk on graphs. This reduces proving \pref{thm:main} to proving the following.

\begin{theorem}\label{thm:goodtimes} If $(G,\rho)$ is a stationary random graph of weakly annealed polynomial growth, then for every $\e > 0$, there is a constant $C > 0$ and an infinite (deterministic) sequence of times $\{t_n\}$ such that $H_{t_n} - H_{t_{n-1}} \leq \frac{C}{t_n}$ and, \[ \Pr\left({\E}\left[d_G(X_0, X_{t_n})^2 \mid (G,\rho)\right] \leq C t_n\right) \geq 1-\e\,. \] \end{theorem}

The proof of the preceding theorem constitutes the bulk of this article. We first show how  \pref{thm:main} follows.

\begin{proof}[Proof of \pref{thm:main}] Observe that by the chain rule for entropy and stationarity of $(G,\rho),$ it follows that for any $t \geq 1$, \begin{equation}\label{entropystationary} \E\left[H_{(G,\rho)}(1,t)-H_{(G,\rho)}(1)\right]=H_{t-1}\,. \end{equation}

There by \pref{thm:goodtimes}, \eqref{entropystationary} and Fatou's Lemma, for any $\e>0$, there is a constant $C > 0$ such that with probability $1-\e$ over the choice of $(G,\rho)$, there exists an infinite sequence of times $\{t_n\}$ (depending on $(G,\rho)$) such that,
\begin{equation}\label{eq:fatou} H_{(G,\rho)}(t_n)+H_{(G,\rho)}(1)-H_{(G,\rho)}(1,t_n) \leq \frac{C}{t_n}     \textrm{ and } \E\left[d_G(X_0,X_{t_n})^2 \mid (G,\rho)\right] \leq Ct_n\,, \end{equation} (for the first inequality notice $\E[H_{(G,\rho)}(t_n)+H_{(G,\rho)}(1)-H_{(G,\rho)}(1,t_n)]=H_{t_n}-H_{t_{n-1}}$).

\medskip

Suppose $h : V_G \to \R$ is harmonic on $G$.  The authors of \cite{benjamini} establish the inequality: For any time $t \geq 1$, \begin{equation}\label{eq:bdky} {\E} \left[ |h(X_0)-h(X_1)| \mid (G,\rho)\right] \leq \sqrt{4[H_{(G,\rho)}(1)+H_{(G,\rho)}(t)-H_{(G,\rho)}(1,t)]\cdot {\E} [h(X_t)^2 \mid (G,\rho)]}\,. \end{equation}
(This inequality is the conjunction of inequality (11) and the first inequality in the proof of Theorem 8 in \cite{benjamini}.)

If the graph $(G,\rho)$ is such that \eqref{eq:fatou} holds and $h : V_G \to \R$ has sublinear growth, then if we consider \eqref{eq:bdky} along the sequence $t=t_n$ and send $n \to \infty$, we conclude that almost surely \[ h(X_0) = h(X_1)\,. \] Now send $\e \to 0$ to conclude that almost surely on $(G,\rho)$ and the random walk $\{X_t\}$, we have $h(X_0) = h(X_1).$

By stationarity, this implies $h(X_t)=h(X_{t+1})$ almost surely for every time $t \geq 1$.  Since $G$ is almost surely connected, we conclude that $h$ must be constant. Therefore \pref{thm:goodtimes} implies that almost surely $(G,\rho)$ does not admit a non-constant harmonic function of sublinear growth. \end{proof}

\subsection{Speed via Euclidean embeddings}

Note that for discrete groups of polynomial growth, significantly stronger results than \pref{thm:goodtimes-intro} are known (giving precise estimates
on the heat kernel).  See, for instance, the work of Hebisch and Saloff-Coste \cite{HS93}.
But those estimates require detailed information about the geometry that is furnished
by Gromov's classification of such groups (in particular, they require the counting measure
to be doubling).  Clearly such methods are unavailable in our setting.

Even when one does not know that the counting measure is doubling,
polynomial growth of a graph $G$ still yields infinitely many radii $r > 0$ at which $|B_x^G(2r)| \leq C |B_x^G(r)|$
for some constant $C > 0$ depending only on the growth rate.  Indeed, locating such scales and
performing geometric arguments that depend only on the local doubling constant underlie Kleiner's remarkable proof of Gromov's theorem \cite{Kleiner10} (see also the quantitative results in \cite{ST10}).
(Somewhat related to the topic of the current paper, the heart of Kleiner's argument lies in establishing
that on any finitely generated group of polynomial growth, the space of harmonic functions
of (fixed) polynomial growth is finite-dimensional.)

We will pursue a related course, but in order to bound the speed of the random walk after $n$ steps,
we require control on the volume growth over $\approx \log n$ scales, corresponding to distances in the interval $[\sqrt{n},n]$.
Polynomial volume growth is certainly not sufficient to find $\log n$ consecutive scales at which
the growth is doubling (uniformly in $n$).  Confronting this difficulty is the major technical challenge we face.

\medskip
\paragraph{Reducing to analysis on finite subgraphs}

In order to establish \pref{thm:goodtimes}, we first invoke the mass transport principle to show that it suffices to examine the random walk restricted to finite subgraphs of $(G,\rho)$.
Let $\mu_G(S) = \sum_{x \in S} \deg_G(x)$ for all subsets $S \subseteq V_G$.

In \pref{sec:mtp}, we argue that it is enough to find an infinite sequence of times $\goodtimes$ and radii $\{r_n : n \in \goodtimes\}$ such that the following three conditions hold for some constant $C$: \begin{enumerate} \item  For every $\e > 0$ and all $n \in \goodtimes$ with $n \geq (1/\e)^4$, \begin{equation}\label{eq:cond1} \E\left[\frac{\mu_G\left(\vphantom{\bigoplus}\left\{ x \in B^G_{\rho}(r_n) : \E[d_G(X_0, X_n)^2 \mid X_0 = x, (G,\rho)] \geq (C\e^{-13}) n\right\}\right)}{\mu_G(B^G_{\rho}(r_n))}\right] \leq \e\,, \end{equation} \item  $H_n - H_{n-1} \leq \displaystyle\frac{C}{n},$ \item
$\E\left[\log \frac{\mu_G(B^G_{\rho}(2r_n))}{\mu_G(B^G_{\rho}(r_n))}\right] \leq C\,.$
\end{enumerate}
It is noteworthy that our application of the mass transport principle
uses the polynomial growth condition; specifically, we need to apply
it at a scale where $\mu_G$ is doubling (see \pref{lem:locals}).

\medskip
\paragraph{Embeddings and martingales}

Let us focus now on condition (1) since it is the difficult one to verify. In order to control the speed of the random walk started at a uniformly random point of $B_{\rho}^G(r_n)$, we construct a family of mappings $\{F_k\}$ from $B_{\rho}^G(r_n)$ into a Hilbert space and use the martingale methods of \cite{NPSS06,DLP13} to derive bounds on the speed. The following statement is a slightly weaker version of \pref{lem:rwcontrol} in \pref{sec:martingales}.

\begin{lemma}\label{lem:intro-rw} Consider a graph $G=(V_G,E_G)$, a finite subset $S \subseteq V_G$, and a family $\{F_k : S \to \cH\}_{k \in \N}$ of $1$-Lipschitz mappings into a Hilbert space. Let $\f  : \N \to \R_+$ be a given function. For $k_0 \in \N$, define the set of pairs \[ \cG(k_0, \f) = \left\{ (x,y) \in S^2 : \textrm{for all $k \geq k_0$},\ d_G(x,y) \geq 8^k \implies \|F_k(x)-F_k(y)\|_{\cH} \geq \frac{8^k}{\f(k)} \right\}\,. \] If $\{Z_t\}$ is the stationary random walk restricted to $S$ (cf. \pref{def:restricted}), then for every $n \geq 1$, \begin{align}
\E\left[d_G(Z_{2n},Z_0)^2 \cdot \1_{\cG(\alpha_n,\f)}(Z_0, Z_{2n})\right] %&\leq 2n + \sum_{k=\alpha_n}^{\inft} 8^{2k+2} \Pr[d_H(Z_{2n},Z_0) \geq  8^k]\,, \\
&\leq 2n + 256 \sum_{k \geq \alpha_n} 8^{2k} \exp\left(\frac{-8^{2k}}{32 n \f(k)^2}\right)\,,\label{eq:intro-rw} \end{align} where $\alpha_n = \lceil\log_8(\sqrt{2n})\rceil$. \end{lemma}

In \pref{sec:embed}, we show how standard tools from metric embedding theory \cite{CKR01,KLMN05} provide a family of maps which are co-Lipschitz at a fixed scale, assuming the growth rate of balls at that scale is small.

\begin{lemma}[Statement of \pref{lem:embed}] \label{lem:intro-embed} For any graph $G=(V_G,E_G)$ and any $k \geq 1$, there is a $1$-Lipschitz map $F_k : V_G \to \ell_2$ such that for all $x,y \in V_G$, it holds that \[ d_G(x,y) \geq 8^k \implies \|F_k(x)-F_k(y)\|_2 \geq \frac{8^k}{128 \left(1+\log \frac{|B^G_x(8^k)|}{|B^G_x(8^{k-1})|}\right)} \] \end{lemma}

It may help to consider now the following special case:  Suppose that the counting measure on $G$ is doubling, i.e. \[ \sup_{x \in V_G, r \geq 0} \frac{|B_x^G(2r)|}{|B_x^G(r)|} \leq O(1)\,. \] In that case, if we use the family $\{F_k\}$ from \pref{lem:intro-embed}, then there is some uniformly bounded function $\f : \N \to \R_+$  in \pref{lem:intro-rw} such that $\1_{\cG(k_0, \f)} \equiv \1$ for all $k_0 \geq 1$. Evaluating the sum in \eqref{eq:intro-rw} immediately yields $\E[d_G(Z_{2n},Z_0)^2] \leq O(n)$, completing our verification of \eqref{eq:cond1}. (Strictly speaking, the stationary measure on $S$ and the measure $\mu_G$ restricted to $S$ are different, but they can be made arbitrarily close by taking $S=B^G_{\rho}(r_n)$ where $r_n$ is chosen so that $S$ is a sufficiently good F{\o}lner set.)

\medskip

In general, polynomial growth does not imply that the counting measure is doubling (and certainly the annealed form introduces even more complexity). Still, using \pref{lem:intro-rw} and \pref{lem:intro-embed} in conjunction, in \pref{sec:choosing} we show that \eqref{eq:cond1} holds at time $n$ (for some radius $r_n \gg n$) if the average profile of growth rates of balls $\{B^G_x(r) \subseteq B_{\rho}^G(r_n)\}$ is sufficiently well-behaved for $r \in [\sqrt{n}, n]$.

Finally, in \pref{sec:SRG}, we argue that the annealed growth condition \eqref{eq:poly-growth} allows us to find an infinite sequence of radii at which the average growth profile is well-behaved (with high probability over the choice of $(G,\rho)$). This is subtle, as we require control on the growth for $\approx \log n$ scales (corresponding to $r \in [\sqrt{n},n]$).\footnote{The Varopoulos-Carne bound suggests we only need control for $\log \log n$ scales corresponding to $r \in [\sqrt{n},\sqrt{n \log n}]$, but the same problem arises.} As mentioned before, one cannot hope to find such a sequence of consecutive scales at which the volume growth is uniformly doubling. Fortunately, the subgaussian tail in \eqref{eq:intro-rw} gives us some flexibility; it will suffice to find a sequence of consecutive scales where the volume growth is not increasing too fast. Once this is established, we can verify \eqref{eq:cond1} along this sequence and confirm \pref{thm:goodtimes}.

\subsection{A deterministic example:  Planar graphs}
\label{sec:planar}

In this section, we present a solution to a question of Benjamini about random walks on planar graphs. It illustrates some of the ideas our main argument and their origins (in K. Ball's notion of Markov type), as well as the reduction of speed questions to the setting of stationary Markov chains on finite subgraphs.

\medskip

Consider again a graph $G=(V_G,E_G)$. For a finite subset $S \subseteq V_G$, define the {\em edge boundary} \[ \partial_E S = \left\{ \vphantom{\bigoplus} \{u,v\} \in E_G : \1_S(u) \neq \1_S(v)\right\}\,, \] and the {\em edge expansion of $S$} for $S \neq \emptyset$: \[ \phi_G(S) = \frac{|\partial_E S|}{\mu_G(S)}\,. \] Say that $G$ is {\em amenable} if $\inf \{ \phi_G(S) :  \textrm{finite } S \subseteq V_G, S \neq \emptyset\} = 0$. Otherwise, say that $G$ is {\em non-amenable.}

Let $\{X_t\}$ denote simple random walk on $G$.  We say that the walk is {\em ballistic} if there is a constant $c > 0$ such that for all $v \in V_G$, \[ \E\left[d_G(X_0,X_t)^2 \mid X_0=v\right] \geq ct^2 \] for all $t \geq 0$. Say that the walk is {\em always somewhere at most diffusive} if there is a constant $c > 0$ such that for all $t \geq 0$, \[ \inf_{v \in V_G} \E\left[d_G(X_0,X_t)^2 \mid X_0=v\right] \leq ct\,. \]

The following result was conjectured by Itai Benjamini.\footnote{It was made by Benjamini at the Erd\"os Centennial in Budapest, July, 2013} It states that for planar graphs, there are no intermediate (uniform) speeds between $\sqrt{t}$ and $t$.

\begin{theorem}\label{thm:itai} Suppose that $G$ is an infinite planar graph with uniformly bounded vertex degrees. Either $G$ is amenable and the random walk is always somewhere at most diffusive, or $G$ is non-amenable and the random walk is ballistic. \end{theorem}

Benjamini suggested this as an analog to the following dichotomy:  Every amenable planar $G$ graph admits arbitrarily large sets $S \subseteq V_G$ such that $|\partial_V S| \leq O(\sqrt{|S|})$, where $\partial_V S = \{ v \in V_G : d_G(v,S)=1 \}$. (This fact was announced by Gromov; see \cite{bowditch95} for a short proof.) Of course, in the non-amenable case, one has a linear isoperimetric profile:  $|\partial_V S| \geq c|S|$ for some $c > 0$ and every $S \subseteq V_G$. Note that \pref{thm:itai} is straightforward in the non-amenable case: If a graph $G$ is non-amenable, then $G$ has spectral radius $\rho < 1$ \cite{Kesten59}, hence the random walk is ballistic (see, e.g. \cite[Prop. 8.2]{Woess00}).

\begin{remark} If one removes the assumption of bounded degrees from \pref{thm:itai}, then for $G$ amenable, it still holds that the random walk is always somewhere at most diffusive (the argument below does not assume any bound on the vertex degrees). But there are non-amenable planar graphs for which the random walk does not have positive speed.  We refer to \cite[Ex 6.56]{LP:book} for a description of the unpublished construction of Angel, Hutchcroft, Nachmias, and Ray.
\end{remark}

For the amenable case, we recall K. Ball's notion of Markov type \cite{Ball92}.

\begin{definition}[Markov type] A metric space $(X,d)$ is said to have {\em Markov type $p \in [1,\infty)$} if there is a constant $M > 0$ such that for every $n \in \mathbb N$, the following holds. For every reversible Markov chain $\{Z_t\}_{t=0}^{\infty}$ on $\{1,\ldots,n\}$, every mapping $f : \{1,\ldots,n\} \to X$, and every time $t \in \mathbb N$, \begin{equation}\label{eq:mtype} \E \left[d(f(Z_t), f(Z_0))^p\right] \leq M^p t \, \E \left[d(f(Z_0), f(Z_1))^p\right]\,, \end{equation} where $Z_0$ is distributed according to the stationary measure of the chain. One denotes by $M_p(X,d)$ the infimal constant $M$ such that the inequality holds. \end{definition}

\begin{definition}[Restricted random walk] \label{def:restricted} Consider a graph $G=(V_G,E_G)$, and let \[N(x) = \{ y \in V_G : \{x,y\} \in E_G \}\] denote the neighborhood of a vertex $x \in V_G$. Fix a finite subset $S \subseteq V_G$. Denote the measure $\pi$ on $S$ by $\pi(x) = \deg_G(x)/\mu_G(S)$. We define {\em the random walk restricted to $S$} as the following process $\{Z_t\}$: For $t \geq 0$, put \[ \Pr(Z_{t+1} = y \mid Z_t = x) = \begin{cases} \frac{|N(x) \setminus S|}{\deg_G(x)} & y = x \\ \frac{1}{\deg_G(x)} & y \in N(x) \cap S \\ 0 & \textrm{otherwise.} \end{cases} \] It is straightforward to check that $\{Z_t\}$ is a reversible Markov chain on $S$ with stationary measure $\pi$. If $Z_0$ has law $\pi$, we say that $\{Z_t\}$ is the {\em stationary random walk restricted to $S$.} \end{definition}

\begin{definition}[Graphic Markov type] Define the {\em graphic Markov type $p$ constant $\edge_p(G)$} of a graph $G=(V_G,E_G)$ as the infimal number $M$ such that for every finite subset $S \subseteq V_G$ and $t\in \N$, \[ \E\left[d_G\left(Z^S_0, Z^S_t\right)^p\right] \leq M^p t\,, \] where $\{Z^S_t\}$ is the stationary random walk restricted to $S$. \end{definition}

\begin{lemma}\label{lem:amenablemt} If $G$ is an amenable graph, then for every time $t \geq 0$, \[ \inf_{v \in V_G} \mathbb E \left[ d_G(X_0,X_t)^2 \mid X_0=v\right] \leq 2\,\left(\edge_2(G)\right)^2 t \leq 2\,M_2(G)^2 t\,. \] \end{lemma}

We will prove this lemma momentarily.  Let us observe first that \pref{thm:itai} follows immediately in conjunction with the next theorem.

\begin{theorem}[\cite{DLP13}] There is a constant $K > 0$ such that $M_2(G) \leq K$ for any planar graph $G$. \end{theorem}

We remark that bounding $\edge_2(G)$ (which is all that is needed to apply \pref{lem:amenablemt}) is somewhat easier than bounding $M_2(G)$; see \pref{cor:threshold} and the remarks thereafter.

\begin{proof}[Proof of \pref{lem:amenablemt}] Fix a subset $S \subseteq V_G$. Let $\{Z_t\}$ denote the stationary random walk restricted to $S$. From the definition of graphic Markov type, for every $t \geq 0$, we have \begin{equation}\label{eq:MT2} \E[d_G(Z_0,Z_t)^2] \leq \left(\edge_2(G)\right)^2 t \E[d_G(Z_0,Z_1)^2] \leq \left(\edge_2(G)\right)^2 t\,. \end{equation}

Note that since $Z_t$ is stationary, it holds that for all $t \geq 0$, we have $\Pr(Z_{t+1}=Z_t) = \phi_G(S)$. Recall that $\{X_t\}$ is the random walk on $G$. If $X_0$ has the law of $Z_0$, then $X_t$ has the law of $Z_t$ conditioned on the event $\{X_0, X_1, \ldots, X_t\} \subseteq S$.

In particular, we can conclude that \begin{equation}\label{eq:leave} \Pr[\{X_0, X_1, \ldots, X_t\} \nsubseteq S \mid X_0 = Z_0] \leq \sum_{t=0}^{t-1} \Pr(Z_{t+1} = Z_t) \leq t \phi_G(S)\,. \end{equation} Hence, \begin{align} \E[d_G(X_t, X_0)^2 \mid X_0 = Z_0] &\leq \E[d_G(Z_t,Z_0)^2] +  \nonumber
 \Pr[\{X_0, X_1, \ldots, X_t\} \nsubseteq S \mid X_0 = Z_0] \cdot t^2
\\ & \leq \left(\edge_2(G)\right)^2 t + t^3 \phi_G(S)\,, \label{eq:amenbound} \end{align} where in the first line we have used the fact that $d_G(X_t,X_0)^2 \leq t^2$ holds with probability one, and in the second line we have employed the bounds \eqref{eq:MT2} and \eqref{eq:leave}.

Now fix a time $t \geq 0$.  Since $G$ is amenable, there exists a choice of $S$ for which $\phi_G(S) \leq (\edge_2(G)/t)^2$. In this case, from \eqref{eq:amenbound} we obtain \[ \E[d_G(X_t, X_0)^2 \mid X_0 = Z_0] \leq 2\, \left(\edge_2(G)\right)^2 t\,. \] Thus certainly the bound holds for some fixed $X_0 \in S$, concluding the proof. \end{proof}

\section{Martingales, embeddings, and growth rates}

Our proof of \pref{thm:goodtimes} involves the construction of embeddings of $(G,\rho)$ into a Hilbert space $\cH$. The embeddings give rise to a family of martingales in $\cH$ whose behavior can be used to control the speed of the random walk in $G$. This section is primarily expository; we review the martingale methods of \cite{NPSS06,DLP13} and a construction of Euclidean embeddings that reflect the local geometry of a discrete metric space at a fixed scale \cite{CKR01,KLMN05}.

\subsection{Control by martingales} \label{sec:martingales}

Consider a finite metric space $(X,d)$. Let $\{Z_t\}$ denote a stationary, reversible Markov chain on $X$
with the property that
\begin{equation}\label{eq:bddjumps}
d(Z_0, Z_1) \leq 1 \quad\textrm{ almost surely.}
\end{equation} 
 Let $\cY$ be a normed space and for a map $f : X \to \cY$, define \[ \|f\|_{\Lip} = \max_{x \neq y \in X} \frac{\|f(x)-f(y)\|_{\cY}}{d(x,y)}\,. \]

The following result is proved in \cite{NPSS06} (see also \cite{LZ94}). A similar decomposition appears already in the work of Kesten \cite{Kesten86} (see the discussion in \cite[Sec. 2]{BP89}) for the special case of percolation clusters in $\Z^d$. A stark difference is that in Kesten's paper, the Markov chain $\{Z_t\}$ already takes values in a subset of $\Z^d$ (and hence the map $f$ does not appear). On the other hand, this means that Kesten only bounds the speed of the walk in the ambient Euclidean metric, whereas we are interested in the speed in the intrinsic metric (which is larger, and hence harder to bound from above).

\begin{lemma}\label{lem:npss}
 Then for every $n \geq 1$, there is a forward martingale $\{A_t\}$ and a backward martingale $\{B_t\}$ such that \begin{enumerate} \item $f(Z_{2n}) - f(Z_0) = A_{n} - B_{n}$ \item For all $t=1,2,\ldots,n$, it holds that \[ \|A_{t}-A_{t-1}\|_{\cY}, \|B_t-B_{t-1}\|_{\cY} \leq 2 \|f\|_{\Lip}\,. \] \end{enumerate} \end{lemma}

For completeness we include the proof. \begin{proof}
 Define the martingales $\{M_s\}^{2n}_{s=0}$ and $\{N_s\}^{2n}_{s=0}$
by $M_0 = f(Z_0)$ and $N_0 = f(Z_{2n})$ and for $0\le s \le t-1$, \begin{align}\label{martingale} M_{s+1} - M_s &\seteq f(Z_{s+1}) - f(Z_s) -\E [f(Z_{s+1}) - f(Z_s) \mid Z_s] \\ \nonumber N_{s+1} - N_s &\seteq f(Z_{t-s-1}) - f(Z_{t-s}) - \E [f(Z_{t-s-1}) -f(Z_{t-s}) \mid Z_{t-s}]. \end{align} Observe that $\{M_s\}$ is a martingale with respect to the filtration induced on $\{Z_0, Z_1,\ldots , Z_{2n}\}$ and $\{N_s\}$ is a martingale with respect to the filtration induced on $\{Z_{2n}, Z_{2n-1},\ldots, Z_0\}$.

For every $1\le s \le 2n-1$ using stationarity gives $\E[f(Z_{s+1}) \mid Z_s] = \E[f(Z_{s-1}) \mid Z_s]$, and it follows that \begin{equation} \label{decomposition} f(Z_{s+1}) - f(Z_{s-1}) = (M_{s+1} - M_{s}) - (N_{2n-s+1} - N_{2n-s}). \end{equation} Now consider the martingales $\{A_t\}_{0\le t \le n}$ and $\{B_t\}_{0\le t \le n}$ given by \begin{align*} A_t& \seteq \sum_{s=0}^{t}M_{2s}-M_{2s-1}\\ B_t & \seteq \sum_{s=0}^{t}N_{2s}-N_{2s-1}. \end{align*} (2) follows the preceding definition and \eqref{martingale},
 along with assumption \eqref{eq:bddjumps}. The proof of (1) is by summing \eqref{decomposition} over $s=1,3,\ldots ,2n-1.$ \end{proof}

Combining \pref{lem:npss} with Azuma's inequality for $\cH$-valued martingales \cite{azuma} yields the following.

\begin{corollary}\label{cor:tail} If $\cH$ is a Hilbert space, then for all $n \geq 1$, \[ \Pr\left(\|f(Z_{2n})-f(Z_0)\|_{\cH} \geq \lambda\right) \leq 4 \exp\left(\frac{-\lambda^2}{32 n \|f\|_{\Lip}^2}\right) \] \end{corollary}

Define the constants \begin{align}\label{eq:alphan} \alpha_n &= \left\lceil \log_8(\sqrt{2n})\right\rceil \\ \beta_n &= \lceil \log_8(2n)\rceil\,.\label{eq:betan} \end{align}

\begin{lemma}\label{lem:rwcontrol} Consider a graph $G=(V_G,E_G)$, a finite subset $S \subseteq V_G$, and a family $\{F_k : S \to \cH\}_{k \in \N}$ of $1$-Lipschitz mappings into a Hilbert space. Let $\f  : \N \to \R_+$ be a given function. For $k_0 \in \N$, define the set of pairs \[ \cG(k_0, \f) = \left\{ (x,y) \in S^2 : \textrm{for all $k \geq k_0$},\ d_G(x,y) \geq 8^k \implies \|F_k(x)-F_k(y)\|_{\cH} \geq \frac{8^k}{\f(k)} \right\}\,. \] If $\{Z_t\}$ is the stationary random walk restricted to $S$ (cf. \pref{def:restricted}), then for every $n \geq 1$, \begin{align*}
\E\left[d_G(Z_{2n},Z_0)^2 \cdot \1_{\cG(\alpha_n,\f)}(Z_0, Z_{2n})\right] %&\leq 2n + \sum_{k=\alpha_n}^{\inft} 8^{2k+2} \Pr[d_H(Z_{2n},Z_0) \geq  8^k]\,, \\
&\leq 2n + 256 \sum_{k=\alpha_n}^{\beta_n} 8^{2k} \exp\left(\frac{-8^{2k}}{32 n \f(k)^2}\right)\,. \end{align*} \end{lemma}

\begin{proof} Use that fact that for a non-negative random variable $X$, we have $\E[X^2] \leq \sum_{k=0}^{\infty} 8^{2k+2} \Pr[X \geq 8^k]$ to write \begin{align*} \E[d_G(Z_{2n},Z_0)^2 \cdot \1_{\cG(\alpha_n,\phi)}(Z_0, Z_{2n})] &\leq \sum_{k=0}^{\beta_n} 8^{2k+2} \Pr[d_G(Z_{2n},Z_0) \geq 8^k \wedge (Z_0,Z_{2n}) \in \cG(\alpha_n,\f)]\\ & \leq \sum_{k=0}^{\beta_n} 8^{2k+2} \Pr\left[\|F_k(Z_0)-F_k(Z_{2n})\|_{\cH} \geq \frac{8^k}{\f(k)}\right] \\ &\leq 2n + \sum_{k=\alpha_n}^{\beta_n} 8^{2k+2} \Pr\left[\|F_k(Z_0)-F_k(Z_{2n})\|_{\cH} \geq \frac{8^k}{\f(k)}\right]\,, \end{align*} where in the first inequality we have used the fact that $d_G(Z_0, Z_{2n}) \leq 2n$ is always true. The desired bound now follows from \pref{cor:tail}. \end{proof}

We remark on one straightforward (but illustrative) application of \pref{lem:rwcontrol}. Following \cite{DLP13}, we say that a metric space $(X,d)$ admits a {\em threshold embedding with distortion $D$ into a Hilbert space $\cH$} if there is a family of $1$-Lipschitz maps $\{ F_k : X \to \cH \}$ such that \begin{equation}\label{eq:threshold} x,y \in X \textrm{ and } d(x,y) \geq 8^k \implies \|F_k(x)-F_k(y)\|_{\cH} \geq \frac{8^k}{D}\,. \end{equation} It is proved in $\cite{DLP13}$ that if such a threshold embedding exists, then $M_2(X,d) \leq O(D)$ (recall the definition of Markov type from \pref{sec:planar}). Bounding the graphic Markov type is substantially easier.

\begin{corollary}\label{cor:threshold} If $G=(V_G, E_G)$ is a graph and $(V_G, d_G)$ admits a threshold embedding into a Hilbert space $\cH$ with distortion $D$, then \[ \edge_2(G) \leq O(D)\,. \] \end{corollary}

\begin{proof} Fix a finite subset $S \subseteq V_G$.  Let $\{Z_t\}$ denote the stationary random walk restricted to $S$. Let $\{F_k : V_G \to \cH\}$ be the claimed threshold embedding. Apply \pref{lem:rwcontrol} to the family $\{F_k|_S\}$ with $\f \equiv D$, in which case $\1_{\cG(\alpha_n,k_0)} \equiv \1$. One concludes that for every $n \geq 1$, \[ \E[d_G(Z_{2n},Z_0)^2] \leq O(Dn)\,. \] Using $d_G(Z_n,Z_0) \leq d_G(Z_{n+1},Z_0) + 1$ yields a similar estimate for odd times, completing the proof. \end{proof}

On the other hand, we will not have a uniform lower bound as in \eqref{eq:threshold} that holds for all pairs $x,y \in X$.

\medskip

\paragraph{Volume growth}

Let $G$ be a graph with vertex set $V_G$. For $x \in V_G$, we recall that $B^G_x(R)$ is the closed $R$-ball around $x$ in the metric $d_G$. Define \begin{equation}\label{eq:phidef} \phi^{G}_x(k) = \log \frac{\left|B^G_x(8^k)\right|}{\left|B^G_x(8^{k-1})\right|}\,. \end{equation}

In the next section, we exibit a family of mappings that reflect the geometry of $G$ well at scale $8^k$ when $\phi^G_x(k)$ is small.

\begin{lemma}\label{lem:embed} For any $k \geq 1$, there is a $1$-Lipschitz map $F_k : V_G \to \ell_2$ such that for all $x,y \in V_G$, it holds that \[ d_G(x,y) \geq 8^k \implies \|F_k(x)-F_k(y)\|_2 \geq \frac{8^k}{128 (1+\phi^G_x(k))} \] \end{lemma}

\subsection{Embeddings and growth rates} \label{sec:embed}

For a metric space $(X,d)$, define $B(x,R) = \{ y \in X : d(x,y) \leq R \}$. We now prove the following generalization of \pref{lem:embed}.

\begin{lemma}\label{lem:single-scale-embedding} If $(X,d)$ is a discrete metric space, then the following holds. For any $\tau > 0$, there is a $1$-Lipschitz mapping $\f_{\tau} : X \to \ell_2$ such that for all $x,y \in X$, \[ d(x,y) \geq \tau \implies \|\f_{\tau}(x)-\f_{\tau}(y)\|_2 \geq \frac{\tau}{128   \log \left(\frac{e |B(x,\frac{5}{8} \tau)|}{|B(x,\frac{1}{8}\tau)|}\right)} \] \end{lemma}

\pref{lem:embed} is a well-known result in metric embedding theory; see, e.g., \cite{KLMN05} where a similar lemma is stated. We provide a proof here for the sake of completeness.

By a simple compactness argument, it suffices to prove \pref{lem:single-scale-embedding} for $X$ finite, which we now assume. Given a probability space $(\Omega,\cB,\mu)$, we use $L^2(\mu)$ to denote the Hilbert space of measureable real-valued random variables with inner product $\langle Y,Z\rangle_{L^2(\mu)} = \E[YZ]$. If $P$ is a partition of $X$, we denote by $P : X \to 2^X$ the map that sends $x \in X$ to the unique set $P(x) \in P$ containing $x$.

\begin{lemma}\label{lem:embed1} For any value $\tau > 0$ and $\e : X \to \R_+$, the following holds.  Let $\cP$ be a random partition of $X$ with the following two properties: \begin{enumerate} \item Almost surely, $\max_{S \in \cP} \diam_{(X,d)}(S) < \tau$. \item For every $x \in X$, \[\Pr\left[\vphantom{\bigoplus}B(x, \e(x) \tau) \subseteq \cP(x)\right] \geq \delta\,. \] \end{enumerate} Then there exists a $1$-Lipschitz mapping $\varphi_{\tau} : X \to \ell_2$ such that for all $x,y \in X$, \[ d(x,y) \geq \tau \implies \|\varphi_{\tau}(x)-\varphi_{\tau}(y)\|_{2} \geq \frac{\sqrt{\delta}}{2} \e(x) \tau\,. \] \end{lemma}

\begin{proof} For every $P \in \supp(\cP)$, let $\{\alpha_S : S \in P\}$ be a sequence of i.i.d. Bernoulli $\{0,1\}$ random variables (independent of $\cP$).

Consider the (random) map $F : X \to \R$ given by \[ F(x) = \alpha_{P(x)} \cdot d(x, X \setminus P(x))\,. \] By construction, $F$ is almost surely $1$-Lipschitz.

Now fix $x,y \in X$ with $d(x,y) \geq \tau$. Note that by assumption (1), $P(x) \neq P(y)$.  Therefore \[ \E |F(x)-F(y)|^2 \geq \e(x)^2 \tau^2 \cdot \Pr[B(x,\e(x) \tau) \subseteq P(x)] \cdot \Pr[\alpha_{P(x)}=1] \cdot \Pr[\alpha_{P(y)}=0] \geq \frac{\delta \e(x)^2 \tau^2}{4}\,. \] Therefore $F : X \to L^2(\mu)$ provides the desired mapping,
where $\mu$ is the law of the random map $F$.   Note that since $X$ is finite, $\mu$ is finitely supported, so one can take
$L^2(\mu)$ as a finite-dimensional Hilbert space.
\end{proof}

In light of \pref{lem:embed}, in order to prove \pref{lem:single-scale-embedding}, it suffices to construct an appropriate random partition. To do so, we employ the method and analysis of \cite{CKR01}.

\begin{lemma}\label{lem:partition} For every $\tau > 0$, there is a random partition $\cP$ satisfying the assumptions of \pref{lem:embed} with $\delta=\frac12$ and \begin{equation}\label{eq:seteps} \e(x) =\left(32 \log \left(\frac{e |B(x,\frac{5}{8} \tau)|}{|B(x,\frac{1}{8}\tau)|}\right)\right)^{-1} \,. \end{equation} \end{lemma}

\begin{proof} Suppose that $|X|=n$ and let $\pi : [n] \to X$ be a uniformly random bijection. Choose $R \in [\frac{\tau}{4}, \frac{\tau}{2})$ uniformly at random.

Let $\cP$ be the random partition constructed by iteratively cutting out the balls $B(\pi(1), R), B(\pi(2), R), \ldots, B(\pi(n),R)$. In other words, $\cP = \{S_1, S_2, \ldots, S_n\}$ where \[ S_i = B(\pi(i),R) \setminus (S_1 \cup \cdots \cup S_{i-1})\,. \]

Fix a number $\e \leq 1/8$ and a point $x \in X$. Let $T \in [n]$ denote the smallest index for which $d(\pi(T), x) \leq \e \tau + R$. Then we have \begin{equation}\label{eq:badevent} \Pr[B(x,\e \tau) \nsubseteq \cP(x)] \leq \Pr[d(\pi(T), x) \ge R-\e\tau]\,. \end{equation} For $y \in X$, define the interval $I(y) = \left[d(x,y) - \e\tau, d(x,y) +  \e \tau\right)$. Note that the bad event $\{ d(\pi(T),x) > R - \e\tau \}$ is the same as the event $\{R \in I(\pi(T)) \}$.

Order the points of $X$ in non-decreasing order from $x$:  $x_0=x,x_1, \ldots, x_n$. Then \eqref{eq:badevent} yields \begin{align} \Pr[B(x,\e \tau) \nsubseteq \cP(x)] &\leq \Pr[R \in I(\pi(T))]  \nonumber \\
 &= \sum_{j=1}^n \Pr\left[R \in I(x_j)\right] \cdot \Pr[x_j = \pi(T) \mid R \in I(x_j)] \nonumber \\
 %&\leq  16\e \sum_{j=1}^n \Pr[x_j = \pi(T) \mid R \in I(x_j)] \label{eq:interval} \\
 &\leq 16\e \sum_{j=|B(x, \tau/8)|+1}^{|B(x, \frac{5}{8}\tau)|} \Pr[x_j = \pi(T) \mid R \in I(x_j)] \label{eq:bounds} \\
 &\leq 16 \e \sum_{j=|B(x, \tau/8)|+1}^{|B(x, \frac{5}{8}\tau)|} \frac{1}{j} \label{eq:harmonic} \\
 &\leq 16 \e \log \left(\frac{e |B(x,\frac{5}{8} \tau)|}{|B(x,\frac{1}{8}\tau)|}\right)\,.\nonumber
\end{align}

Inequality  \eqref{eq:bounds} arises from the fact that the length of $I(x_j)$ is $4\e\tau$ and $R$ is chosen uniformly from an interval of length $\tau/4$ and  that if $d(x,y) \leq \tau/8$ or $d(x,y) > \frac{5}{8} \tau$, then $\Pr(R\in I(y)) = 0$. Finally, to confirm \eqref{eq:harmonic}, note that \[R \in I(x_j) \implies R > d(x,x_j) - \e\tau \implies R > d(x,x_i) - \e\tau \textrm{ for } i \leq j\,.\] In particular, conditioned on $R \in I(x_j)$, the event $x_j = \pi(T)$ can only happen if $x_j$ is chosen first from $\{x_1, \ldots, x_j\}$ in the permutation $\pi$.

Setting $\e$ as in \eqref{eq:seteps} completes the proof. \end{proof}

\section{Diffusive estimates} \label{sec:folner}

In order to apply the techniques of the preceding section, we need to reduce our main diffusive estimate (\pref{thm:goodtimes}) to a statement about the random walk restricted to finite subgraphs in $(G,\rho)$. In \pref{sec:mtp}, we use the mass transport principle to show that it suffices to control the speed of the random walk on an appropriate sequence of balls $\left\{B_{\rho}^G(r_n)\right\}$ in $G$.

In \pref{sec:choosing}, we argue that this is possible, conditioned on $(G,\rho)$, as long as there are good enough bounds on the average growth rate of balls $\{B^G_x(r) \subseteq B_{\rho}^G(r_n)\}$, where the average is taken over the stationary measure of the random walk restricted to $B_{\rho}^G(r_n)$. Finally, in \pref{sec:SRG}, we show that the weakly annealed polynomial growth property shows yields an infinite sequence of radii $\{r_n\}$ such that the average growth is controlled with high probability over the choice of $(G,\rho)$. This allows us to complete the proof of \pref{thm:goodtimes}.

\subsection{The mass transport principle} \label{sec:mtp}

We now return to the setting where $(G,\rho)$ is a stationary random graph with vertex set $V_G$. For a subset $S \subseteq V_G$, define $\mu_G(S) = \sum_{x \in S} \deg_G(x)$.

In order to establish \pref{thm:goodtimes}, we employ an unpublished result of Russ Lyons  that every stationary random graph of (weakly) annealed subexponential growth is actually a reversible random graph.  For completeness, we indicate a proof at the end of this section.

In particular, we can assume that $(G,\rho)$ satisfies a mass transport principle (see, e.g., the extensive reference \cite{aldous-lyons}
or the discussion in \cite{BC12}): For every positive functional $F(G,\rho,x)$, it holds that \begin{equation}\label{eq:mtp} \E\left[\frac{1}{\deg_G(\rho)} \sum_{x \in V_G} F(G,\rho,x)\right] = \E\left[\frac{1}{\deg_G(\rho)} \sum_{x \in V_G} F(G,x,\rho)\right]\,. \end{equation}

Consider an event $\cA$ in $\cG_{\bullet}$ (depending only on the isomorphism classes of finite rooted subgraphs).

\begin{lemma}\label{lem:locals} For any $R \geq 1$, it holds that, \[ \E\left[\frac{\mu_G(B^G_{\rho}(R))}{\mu_G(B^G_{\rho}(2R))} \cA(G,\rho)\right] \leq \E\left[\frac{\mu_G\left(\left\{\vphantom{\bigoplus} x \in B^G_{\rho}(R) : \1_{\cA}(G,x)\right\}\right)}{\mu_G(B^G_{\rho}(R))}\right]\,. \]
\end{lemma}

\begin{proof} Define a mass transportation: \[ F(G,\rho,x) = \deg_G(\rho) \frac{\deg_G(x)}{\mu_G(B^G_{\rho}(R))} \1_{B^G_{\rho}(R)}(x) \,\1_{\cA}(G,x)\,. \] Observe that, \begin{align*} \E\left[\frac1{\mu_G(B^G_{\rho}(R))} \sum_{x \in B^G_{\rho}(R)} \deg_G(x) \1_{\cA}(G,x)\right] &= \E\left[\frac{1}{\deg_G(\rho)} \sum_{x \in V_G} F(G,\rho,x) \right] \\ &\stackrel{\mathclap{\eqref{eq:mtp}}}{=} \,\E\left[\frac{1}{\deg_G(\rho)} \sum_{x \in V_G} F(G,x,\rho)\right] \\ &= \E\left[\sum_{x \in B^G_{\rho}(R)} \frac{\deg_G(x)}{\mu_G(B^G_x(R))} \1_{\cA}(G,\rho)\right] \\ &\geq \E\left[\frac{\mu_G(B^G_{\rho}(R))}{\mu_G(B^G_{\rho}(2R))} \cA(G,\rho)\right], \end{align*} where the last line follows from the fact that $x \in B^G_{\rho}(R) \implies \mu_G(B^G_x(R)) \leq \mu_G(B^G_{\rho}(2R))$. \end{proof}

The following theorem,
along with the mass transport principle, implies \pref{thm:goodtimes}. Its proof occupies \pref{sec:choosing} and \pref{sec:SRG}.

\begin{theorem}\label{thm:folner-version} Suppose that $(G,\rho)$ is a stationary random graph of weakly annealed polynomial growth. Then there is a constant $C$ depending only on the growth constants of $(G,\rho)$ (cf. \eqref{eq:poly-growth}) and an infinite (deterministic) sequence of times $\goodtimes$ and radii $\{r_n : n \in \goodtimes\}$ such that the following conditions hold: \begin{enumerate} \item \label{cond1} For every $\e > 0$ and all $n \in \goodtimes$ with $n \geq (1/\e)^4$, it holds that, \[ \E\left[\frac{\mu_G\left(\vphantom{\bigoplus}\left\{ x \in B^G_{\rho}(r_n) : \E[d_G(X_0, X_n)^2 \mid X_0 = x, (G,\rho)] \geq (C\e^{-13}) n\right\}\right)}{\mu_G(B^G_{\rho}(r_n))}\right] \leq \e\,, \] \item \label{cond2} $\left(H_n - H_{n-1}\right)n \leq C,$ \item \label{cond3}
$\E\left[\log \frac{\mu_G(B^G_{\rho}(2r_n))}{\mu_G(B^G_{\rho}(r_n))}\right] \leq C\,.$
\end{enumerate} \end{theorem}

We finish off this section with the proof of  \pref{thm:goodtimes}.

\begin{proof}[Proof of \pref{thm:goodtimes}] Fix $\delta > 0$ and apply \pref{thm:folner-version}. Applying Markov's inequality to (3) yields \begin{align*}
\Pr\left[\log \frac{\mu_G(B^G_{\rho}(2r_n))}{\mu_G(B^G_{\rho}(r_n))} \leq \frac{C}{\delta}\right] &\geq 1-\delta\,. \end{align*}

We now lower bound the probability that the random walk started from the root is at most diffusive.  Note that \pref{thm:folner-version} asserts
 this for the majority of the points in $B^G_{\rho}(r_n).$ To transfer this to the root, we use the mass transport principle.

To apply \pref{lem:locals}, we define the set $\cA$ of rooted graphs such that \[\mathbf{1}(\cA)=\mathbf{1}\left(\left\{ (G,\rho) : {\E}_{X_n}[d_G(\rho, X_n)^2 ] \geq C' n\right\}\right)\] for some constant $C'$ which is specified below.
Using \pref{lem:locals} in conjunction with \eqref{cond1} yields $$e^{-C/\delta}\Pr\left(\left\{\log \frac{\mu_G(B^G_{\rho}(2r_n))}{\mu_G(B^G_{\rho}(r_n))} \le \frac{C}{\delta}\right\} \wedge \cA \right)\le \e. $$ Thus by union bound, \begin{align*} \Pr(\cA)  \le e^{c/\delta} \e +\Pr\left(\log \frac{\mu_G(B^G_{\rho}(2r_n))}{\mu_G(B^G_{\rho}(r_n))} \ge \frac{C}{\delta}\right) \le e^{c/\delta} \e+ \delta. \end{align*}

Choosing $\e = \delta e^{-C/\delta},$ yields that for some $C' \leq e^{14C/\delta}$, and for all $n \in \goodtimes$ sufficiently large, \[ \Pr\left(\vphantom{\bigoplus}\E[d_G(X_0, X_n)^2 \mid (G,\rho)] \geq C' n\right) \leq \delta + e^{C/\delta} \delta e^{-C/\delta} \leq 2\delta\,. \]

Therefore it holds that for all $n \in \goodtimes$ sufficiently large, $H_n - H_{n-1} \leq \frac{C}{n}$ and, \[ \Pr\left[\E[d_G(X_0,X_n)^2 \mid (G,\rho)] \geq C' n \right] \leq 3\delta\,, \] yielding the desired result. \end{proof}

\subsubsection{Subexponential growth and reversibility}
We now prove the following unpublished result of Russ Lyons.

Recall that $(G,\rho)$ is stationary if $(G,X_0) \stackrel{\textrm{law}}{=} (G,X_1)$
where $\{X_n\}$ is the random walk on $G$ with $X_0=\rho$.
The random graph $(G,\rho)$ is said to be {\em reversible} if $(G,X_0,X_1) \stackrel{\textrm{law}}{=} (G,X_1,X_0)$.

\begin{theorem}
If $(G,\rho)$ is a stationary random graph such that
\begin{equation}\label{eq:subexp}
\lim_{n \to \infty} \frac{\E[\log |B_{\rho}^G(n)|]}{n} \to 0\,,
\end{equation}
then $(G,\rho)$ is reversible.
\end{theorem}

This result was proved earlier in \cite{BC12} with the additional assumption that $\deg_G(\rho) \leq O(1)$
almost surely.

\begin{proof}
We will borrow heavily from \cite[Ch. 4]{BC12}.  The reader
is encouraged to consult that paper for more detailed explanations.
Let $\mu_{\rightarrow}$ and $\mu_{\leftarrow}$ denote the laws of
$(G,X_0,X_1)$ and $(G,X_1,X_0)$, respectively.  For a fixed graph $G_0=(V_0,E_0)$ and $\{x,y\} \in E_0$,
we denote the Radon-Nikodym derivative
\[
\Delta(G_0, x,y) \seteq \frac{d\mu_{\leftarrow}}{d\mu_{\rightarrow}} (G_0,x,y)\,.
\]
One can extend this to pairs $x,y \in V_0$ which are not necessarily adjacent:  Consider any path $x=x_0,x_1,\ldots,x_n=y$ and define
\[
\Delta(G_0, x,y) \seteq \prod_{i=0}^{n-1} \Delta(G_0,x_i,x_{i+1})\,.
\]
This value is independent of the path between $x$ and $y$ (see \cite[Lem. 4.2]{BC12}; this is a manifestation
of the fact that a cycle and its reverse have the same probability under random walk on a graph).

Note that, because of this, for pairs $x,y \in V_0$ such that $\Pr[X_n = y \mid X_0 = x] > 0$, it holds that
\begin{equation}\label{eq:rd}
\Delta(G_0, x,y) = \frac{d\mu_{n,\leftarrow}}{d\mu_{n,\rightarrow}} (G_0,x,y)\,,
\end{equation}
where $\mu_{n,\rightarrow}$ and $\mu_{n,\leftarrow}$ are the laws of $(G,X_0,X_n)$ and $(G,X_n,X_0)$, respectively.
(This equality only makes sense up to sets of $\mu_{\leftarrow}$-measure zero.)

\medskip

One has $\E[\Delta(G,X_0,X_1)] = 1$,
and moreover Jensen's inequality shows that
\begin{equation}\label{eq:equivalence}
\E[\log \left(\Delta(G,X_0,X_1)\right)] \geq 0 \iff \Delta(G,X_0,X_1)=1 \textrm{ a.s.} \iff (G,X_0,X_1) \textrm{ is reversible.}
\end{equation}
Let $\cG_{\bullet\bullet}$ denote the set of isomorphism classes
of bi-rooted graphs.  Then
for any Borel set $A \subseteq \cG_{\bullet \bullet}$ and $n \geq 0$, stationarity yields
\begin{equation}\label{eq:probs}
\Pr\left[(G,X_0,X_n) \in A\right] = \Pr\left[(G,X_n,X_{2n}) \in A\right] \geq \Pr(X_{2n}=X_0 \mid X_n) \Pr\left[(G,X_n,X_0) \in A\right]\,.
\end{equation}
Let $p_G^n$ denote the $n$-step transition kernel in $G$.
Using \eqref{eq:probs} in \eqref{eq:rd} implies that almost surely:
\begin{equation}\label{eq:xoon}
\Delta(G,X_0,X_n) \geq \Pr\left(X_{2n}=X_0 \mid X_n, G\right) = p_G^n(X_n,X_0) = p_G^n(X_0,X_n) \frac{\deg_G(X_0)}{\deg_G(X_n)}\,.
\end{equation}

Therefore almost surely,
\begin{align*}
\E[\log \Delta(G,X_0,X_n) \mid G, X_0] &\geq \sum_{y} p_G^n(X_0,y) \left(\log p_G^n(X_0,y)+ \log \frac{\deg_G(X_0)}{\deg_G(y)}\right) \,.
\end{align*}
Note that \eqref{eq:subexp} implies $\E[\log \deg_G(\rho)] < \infty$.
Therefore taking expectations and
again employing stationarity yields
\begin{equation}\label{eq:entropy}
\E[\log \Delta(G,X_0,X_n)] \geq -\E\left[H\left(X_n \mid (G,X_0)\right)\right] \geq -\E\left[\log |B_{X_0}^G(n)|\right],
\end{equation}
where $H(\cdot \mid (G,X_0))$ denotes the Shannon entropy conditioned on $(G,X_0)$.

Using again the cycle property \cite[Lem. 4.2]{BC12},
it holds that almost surely, for all $n \geq 0$,
\[
\log \left(\Delta(G,X_0,X_n)\right) = \sum_{t=0}^{n-1}\log \left(\Delta(G,X_t,X_{t+1})\right)\,.
\]
From \eqref{eq:xoon} and the fact that $\E[\log \deg_G(\rho)] < \infty$, we have
$\E |\log \Delta(G,X_0,X_1)| < \infty$.  Thus
using stationarity once more and combining this with \eqref{eq:entropy} yields
\[
\E[\log \Delta(G,X_0,X_1)] = \lim_{n \to \infty} \frac{\E[\log \Delta(G,X_0,X_n)]}{n} \geq \lim_{n \to \infty} \frac{-\E[\log |B_{X_0}^G(n)|]}{n} = 0\,,
\]
implying that $(G,X_0,X_1)$ is reversible (recall \eqref{eq:equivalence}).
\end{proof}

\subsection{Choosing a good F{\o}lner set} \label{sec:choosing}

Fix a rooted graph $(G,\rho)$ with vertex set $V_G$. The most difficult part of the proof of \pref{thm:folner-version} is verifying \eqref{cond1}. Toward this end, we will employ  \pref{lem:rwcontrol} and \pref{lem:embed} to control the random walk restricted to a subset of the vertices in $G$ whenever the local growth rates are sufficiently well-behaved. Consider a finite subset $S \subseteq V_G$.

Let $\{Z_t\}$ denote the stationary random walk restricted to $S$ (recall \pref{def:restricted}), and let $\pi$ denote the corresponding stationary measure. Recall the definition of $\phi^G_x(k)$ from \eqref{eq:phidef}. For $k,k_0 \in \N$, define the numbers \begin{align*} \bar \phi_S(k) &= \sum_{x \in S} \pi(x) \phi^G_x(k) \\ S_{\lambda}(k) &= \left \{ x \in S : \phi^G_x(k) \leq \lambda \bar \phi_S(k) \right\} \\ S^{\uparrow}_{\lambda}(k_0) &= \bigcap_{k\geq k_0} S_{\lambda 2^{k-k_0}}(k)\,. \end{align*}

Note that by Markov's inequality and a geometric summation, we have \begin{equation}\label{eq:upbound} \pi(S^{\uparrow}_{\lambda}(k_0)) \geq 1 - 2/\lambda\,. \end{equation}

\begin{lemma}\label{lem:many-scales} For all $n \geq 1$, it holds that \begin{equation}\label{eq:many-scales} \E\left[d_G(Z_{2n}, Z_0)^2 \cdot \1_{S^{\uparrow}_{\lambda}(\alpha_n)}(Z_0)\right] \leq 2n+   256 \sum_{k=\alpha_n}^{\beta_n} 8^{2k} \exp\left(\frac{-8^{2k}}{\lambda^2 4^{k-\alpha_n+10} (1+\bar \phi_S(k))^2 n}\right)\,. \end{equation} \end{lemma}

\begin{proof} Apply \pref{lem:rwcontrol} using the family of mappings that arises from applying \pref{lem:embed} to $G$, and with the functional $\f(k) = 128(1+\lambda \bar \phi_S(k))$. \end{proof}

Therefore control on $\bar \phi_S(k)$ for $k \in \{\alpha_n, \ldots, \beta_n\}$ yields control on the speed of $\{Z_t\}$. Let us define, for $r \geq 1$, \[ \bar{\phi}^G_{\rho,r}(k) \defeq \bar{\phi}_{B_{\rho}^G(r)}(k)\,, \] and chose $S = B_{\rho}(r)$ for some $r \geq 0$.

\begin{definition}[Tempered growth] \label{def:fertile} We say that a triple $(n,\lambda,r)$ is {\em tempered in $(G,\rho)$} if \begin{equation}\label{eq:fertile} \bar \phi^G_{\rho,r}(k) \leq \lambda 2^{k-\alpha_n} \textrm{ for } k \in \{\alpha_n, \ldots, \beta_n\}\,. \end{equation} For ease of reference, we recall the definitions from \eqref{eq:alphan}--\eqref{eq:betan}: \begin{align*} \alpha_n &= \left\lceil \log_8(\sqrt{2n})\right\rceil \\ \beta_n &= \lceil \log_8(2n)\rceil\,. \end{align*} \end{definition}

\begin{lemma}\label{lem:admissible} For all $n \geq 1$ and $\lambda \geq 2$ the following holds. If $(n,\lambda,r)$ is tempered in $(G,\rho)$ and $S=B_{\rho}(r)$, then \[\E\left[d_G(Z_{2n}, Z_0)^2 \cdot \1_{S^{\uparrow}_{\lambda}(\alpha_n)}(Z_0)\right] \leq O(\lambda^{12}) n\,. \] \end{lemma}

\begin{proof} To see this, apply \pref{lem:many-scales} and note that $8^{2\alpha_n} \leq O(n)$, hence \begin{align*} \sum_{k=\alpha_n}^{\beta_n} 8^{2k} \exp\left(\frac{-8^{2k}}{\lambda^2 4^{k-\alpha_n+10} (1+\bar \phi_S(k))^2 n}\right)&\le O(n)\sum_{k=\alpha_n}^{\beta_n} 8^{2k-2\alpha_n} \exp\left(\frac{-8^{2k-2\alpha_n}}{\lambda^2 4^{k-\alpha_n+10} (1+\lambda 2^{k-\alpha_n})^2}\right)\\ &\le O(n)\sum_{j=0}^{\infty} 8^{2j} \exp\left(\frac{-4^{j}}{9 \cdot 4^{10} \lambda^4 }\right)\\ &\leq O(\lambda^{12}) n\,.\qedhere \end{align*} \end{proof}

\medskip

To compare the (unrestricted) random walk $\{X_t\}$ on $G$ to the walk $\{Z_t\}$ restricted to $B^G_{\rho}(r)$, we will choose some $r \geq 0$ satisfying \begin{equation}\label{good3} \frac{\mu_G\left(B^G_{\rho}(r) \setminus B^G_{\rho}(r-2n)\right)}{\mu_G(B^G_{\rho}(r))} \leq \frac{1}{4\lambda}\,. \end{equation} In particular, this implies that for $\lambda \geq 1$, \begin{align} \pi\left(B^G_{\rho}(r) \setminus B^G_{\rho}(r-2n)\right) &\leq \frac{\mu_G(B^G_{\rho}(r) \setminus B^G_{\rho}(r-2n))}{\mu_G(B^G_{\rho}(r))-\mu_G\left(B^G_{\rho}(r) \setminus B^G_{\rho}(r-2n)\right)}\nonumber \\ &\stackrel{\mathclap{\eqref{good3}}}{\leq} \frac{1/4\lambda}{1-1/4\lambda} \nonumber \\ &\leq \frac{1}{2\lambda}\,. \label{good2} \end{align}

\begin{definition}[Insulation] \label{def:protected} We say that a triple $(n,\lambda,r)$ is {\em insulated in $(G,\rho)$} if \eqref{good3} holds. \end{definition}

Our final choice of $(n,\lambda,r)$ will satisfy some additional constraints,
 hence a complete description of the requirements is postponed to the next section.
However we already have the following.

\begin{lemma}\label{lem:unrestricted} For every $\lambda \geq 2$ the following holds. For any rooted graph $(G,\rho)$, if $(n,\lambda,r)$ is tempered and insulated in $(G,\rho)$, then for $S=B^G_{\rho}(r)$, it holds that \[\pi\left(\left\{\vphantom{\bigoplus}x \in S: \E[d_G(X_{2n},X_0)^2 \mid X_0=x] \geq  \lambda^{13} n\right\}\right) \leq O\left(\frac{1}{\lambda}\right)\,. \] \end{lemma}

\begin{proof} Note that by \pref{lem:admissible}, and Markov's inequality, \[\pi\left(\left\{x \in S_{\lambda}^{\uparrow}(\alpha_n):\E[d_G(Z_{2n},Z_0)^2 \mid X_0=x]\geq \lambda^{13} n\right\} \right) \leq O\left(\frac{1}{\lambda}\right)\,. \]

Now, unless $x\in S\setminus B^G_{\rho}(r-2n)$, one can easily couple $\{X_0, X_1, \ldots, X_{2n}\}$ and $\{Z_0, Z_2, \ldots, Z_{2n}\}$ conditioned on $X_0=Z_0=x$. Thus using the fact that $(n,\lambda,r)$ is tempered in $(G,\rho)$, along with \eqref{good2} and \pref{eq:upbound}, we have \begin{align*} \pi\left(\left\{\vphantom{\bigoplus}x : \E[d_G(X_{2n},X_0)^2 \mid X_0=x] \geq  \lambda^{13} n\right\}\right)& \le  \pi\left(S\setminus S_{\lambda}^{\uparrow}(\alpha_n)\right)+O\left(\frac{1}{\lambda}\right)+ \pi (S\setminus B_{\rho}^G(\rho,r-2n)),\\ & \le O\left(\frac{1}{\lambda}\right)\,.\qedhere \end{align*}
\end{proof} Combining the preceding lemma with \eqref{good3} gives us the following.

\begin{corollary}\label{cor:unrestricted} There is a constant $\kappa > 0$ such that for every $\lambda \geq 2,$  if $(n,\lambda,r)$ is tempered and insulated in $(G,\rho)$, then
 \[\frac{\mu_G\left(\left\{\vphantom{\bigoplus}x \in B^G_{\rho}(r) : \E[d_G(X_{2n},X_0)^2 \mid X_0=x] \geq \lambda^{13} n\right\}\right)}{ \mu_G(B^G_{\rho}(r))}
\leq \frac{\kappa}{\lambda} \,. \] \end{corollary}

\subsection{Multi-scale control of growth functionals} \label{sec:SRG}

In this section we find scales which simultaneously satisfy all the criterion in  \pref{thm:folner-version}.

We begin with the following  observation.
For a function $\phi : \N \to \R$, an integer $\ell \geq 1$, let \begin{align*} \theta(\ell) &= \sum_{k=\ell}^{3\ell} \phi(k) 2^{\ell-k}\,.
\end{align*}

Then an elementary geometric summation yields \begin{equation}\label{eq:sumbound} \sum_{\ell=h}^{2h} \theta(\ell) \leq 2 \sum_{k=h}^{5h} \phi(k)\,. \end{equation}

Define now the quantity \begin{align*} \theta^G_{\rho,r}(\ell) &= \sum_{k=\ell}^{3\ell} \bar \phi^G_{\rho,r}(k) 2^{\ell-k}\,,
\end{align*} recalling that $\bar \phi^G_{\rho,r}(k)$ is the average of $\phi^G_x(k)$ over the stationary measure of the random walk restricted to $B^G_{\rho}(r)$. Note that \[ \theta^G_{\rho,r}(\ell)<\lambda \implies \bar \phi^G_{\rho,r}(k) < \lambda 2^{k-\ell} \textrm{ for } k \in \{\ell,\ell+1,\ldots,3\ell\}\,. \] In particular, recalling \pref{def:fertile}, \begin{equation}\label{eq:fertility} \theta^G_{\rho,r}(\ell)<\lambda \implies (n,128\lambda,r) \textrm{ is tempered in $(G,\rho)$ for $n \in [8^{2\ell},8^{2\ell+2}]$}\,. \end{equation}

Consider now a stationary random graph $(G,\rho)$. For $k \geq 1$, define \[ \hat \psi(k) = \E\left[\log |B^G_{\rho}(8^k)|\right]\,. \] If $(G,\rho)$ has weakly annealed polynomial growth \eqref{eq:poly-growth}, then there is a number $s > 0$ such that \begin{equation}\label{eq:poly-growth1} \hat \psi(k) \leq sk\,. \end{equation}

We now fix a number $k_0 \geq 3$, and try to locate triples $(n,\lambda,r)$ with $n > 8^{k_0}$ that are tempered in $(G,\rho)$ with high probability. To ensure simultaneous occurrence of the many conditions required, we define
\begin{align*} \Psi_{(G,\rho)}(k_0) = \sum_{k=9k_0}^{10k_0}  \left[\log \frac{\mu_G(B^G_{\rho}(8^{k+2}))}{\mu_G(B^G_{\rho}(8^k))}\right. &+ \left(\sum_{r \in I(k)}\log \frac{\mu_G(B^G_{\rho}(r))}{\mu_G(B^G_{\rho}(r-8^{4k_0+3}))}\right.\\ &\quad+\left.\left.\frac{1}{k_0 |I(k)|} \sum_{\ell=k_0}^{2k_0}\left[\theta^G_{\rho,r}(\ell) +\left\{\sum_{n=8^{2\ell}}^{8^{2\ell+2}} H_{(G,\rho)}(2n)-H_{(G,\rho)}(2n-1)\right\}\right]\right)\right]\,, \end{align*} where $I(k) = \left\{8^{k}+8^{4k_0+3},8^{k}+2\cdot 8^{4k_0+3}, 8^k+3\cdot 8^{4k_0+3},\ldots ,8^{k+1}-8^{4k_0+3}\right\}.$

Observe that from \eqref{eq:sumbound}, for any $r \in [9k_0,10k_0]$ we have, \[ \sum_{\ell=k_0}^{2k_0} \theta^G_{\rho,r} (\ell) \leq 2 \sum_{\ell=k_0}^{5k_0} {\bar \phi}^G_{B^G_{\rho}(r)}(\ell) \leq 2 \log |B^G_{\rho}(8^{5k_0}+r)| \leq 2 \log |B^G_{\rho}(8^{11k_0})|\,. \]

The sum in braces is bounded by $H_{(G,\rho)}(2\cdot 8^{4k_0+2})$ which is at most $\log  |B^G_{\rho}(2\cdot 8^{4k_0+2})|$. The first two terms sum telescopically to at most $4 \log \mu_G(B^G_{\rho}(8^{10k_0+2})).$ Putting everything together, we arrive at \begin{equation}\label{finalbound} \Psi_{(G,\rho)}(k_0) \leq 4 \log \mu_G(B^G_{\rho}(8^{10k_0+2})) + 3 \log |B^G_{\rho}(8^{11k_0})| \leq 11\log |B^G_{\rho}(8^{11k_0})|\,. \end{equation} Note that we use the trivial bound $\mu_G(B^G_{\rho}(8^{10 k_0+2})) \le  |B^G_{\rho}(8^{11k_0})|^2$ and the fact that $k_0 \geq 3$.

The growth assumption \eqref{eq:poly-growth} now implies that \[ \gamma \defeq \E[\Psi_{(G,\rho)}(k_0)] \leq 121 k_0 s\,. \]

Thus there must exist numbers $(k,r,\ell,n)$ with $k\in [9k_0,10k_0]$ and $r \in [8^k, 8^{k+1}]$ such that \begin{align} \E\left[\log \frac{\mu_G(B^G_{\rho}(8^{k+2}))}{\mu_G(B^G_{\rho}(8^k))}\right] &\leq \frac{4\gamma}{k_0} \leq O(s)\,, \label{eq:g1}\\ \E\left[\log \frac{\mu_G(B^G_{\rho}(r))}{\mu_G(B^G_{\rho}(r-8^{4k_0+3}))}\right ]& \leq \frac{4\gamma}{8^{5k_0-3}}\,, \label{eq:g2} \end{align} and there are similarly $\ell\in [k_0,2k_0]$ and \begin{equation}\label{eq:nell} n \in [8^{2\ell}, 8^{2\ell+2}] \end{equation} such that \begin{align} \E\left[\theta^G_{\rho,r}(\ell)\right] &\leq \frac{4\gamma}{k_0} \leq O(s) \label{eq:g3}\\ H_{2n}-H_{2n-1}=\E\left[H_{(G,\rho)}(2n)-H_{(G,\rho)}(2n-1) \right] &\leq \frac{4\gamma}{k_0 8^{2\ell}} \leq O(s/n)\label{eq:g4}\,, \end{align}
With the above preparation, we are now ready to finish the proof of \pref{thm:folner-version}.

\begin{proof}[Proof of \pref{thm:folner-version}] For every $k_0 \geq 3$, we obtain a quadruple $\left(k(k_0),r(k_0),\ell(k_0),n(k_0)\right)$ satisfying \eqref{eq:g1}--\eqref{eq:g2} and \eqref{eq:g3}--\eqref{eq:g4}. Fix an infinite and strictly increasing sequence of values $\{k_1, k_2, \ldots\}$ so that the sequence of times \[\goodtimes = \{2 \cdot n(k_i) : i=1,2,\ldots\}\,\] is also strictly increasing.

For $2 \cdot n(k_i) \in \goodtimes$, define $r_{2 n(k_i)}=r(k_i)$. Let $\{k_i, \ell_i, 2 n_i\}$ denote the sequence $\{k(k_i), \ell(k_i), 2 n(k_i)\}$. Inequalities \eqref{eq:g1} and \eqref{eq:g4} show that $\goodtimes$ and $\{r_n : n \in \goodtimes\}$ satisfy conditions \eqref{cond2} and \eqref{cond3} of \pref{thm:folner-version} for some constant $C > 0$. It remains to verify condition \eqref{cond1}.

Toward this end, consider some $\e > 0$. From \eqref{eq:g2} and \eqref{eq:g3}, for every $i=1,2,\ldots$, we can choose constants $c \leq O(s/\e)$ and $b_i \leq O(n_i^{-2/3}/\e)$ such that the event \[ \cE_i = \left\{\theta^G_{\rho,r_{n_i}}(\ell_i) < c\right\} \wedge \left\{\log \frac{\mu_G(B^G_{\rho}(r_{n_i}))}{\mu_G(B^G_{\rho}(r_{n_i}-8^{4k_i+3}))} < b_i\right\} \] has $\Pr(\cE_i) \geq 1- \tfrac12 \e$.

Note that from \eqref{eq:fertility} and the choice \eqref{eq:nell}, we know that for $i=1,2,\ldots$, \[ \cE_i \implies \left(n_i, 128 c, r_{n_i}\right) \textrm{ is tempered in } (G,\rho)\,. \]

Define $\lambda = 2\kappa/\e$, where $\kappa$ is the constant from \pref{cor:unrestricted}. Then since $2n_i \leq 8^{2\ell_i+3} \leq 8^{4k_i+3}$, it holds that \begin{align*} \log \frac{\mu_G(B^G_{\rho}(r_{n_i}))}{\mu_G(B^G_{\rho}(r_{n_i}-8^{4k_i+3}))} < \frac{1}{4\lambda} &\implies \frac{\mu_G\left(B^G_{\rho}(r_{n_i}) \setminus B^G_{\rho}(r_{n_i}-8^{4k_i+3}) \right)}{\mu_G(B^G_{\rho}(r_{n_i})} < \frac{1}{4\lambda} \\ &\implies \frac{\mu_G\left(B^G_{\rho}(r_{n_i}) \setminus B^G_{\rho}(r_{n_i}-2{n_i}) \right)}{\mu_G(B^G_{\rho}(r_{n_i})} < \frac{1}{4\lambda}\,, \\ &\implies (n_i, \lambda, r_{n_i}) \textrm{ is insulated in } (G,\rho)\,. \end{align*} where the first inequality uses $\log(1-x) \leq -x$ for $x \in [0,1)$. Therefore, \[ \cE_i \wedge \left\{b_i < \frac{1}{4\lambda}\right\} \implies (n_i,\lambda, r_{n_i}) \textrm{ is insulated in } (G,\rho)\,. \] Note that $b_i < 1/(4\lambda)$ occurs for all $k_i$ sufficiently large (since $n_i \to \infty$ as $k_i \to \infty$).

We can thus apply \pref{cor:unrestricted} to conclude that if $\cE_i$ occurs and $k_i$ is sufficiently large, then there is a constant $C' > 0$ such that \[ \frac{\mu_G\left(\left\{\vphantom{\bigoplus}x \in B^G_{\rho}(r_{n_i}) : \E[d_G(X_{2n_i},X_0)^2 \mid X_0=x, (G,\rho)] \geq (C' \e^{-13}) n_i\right\}\right)}{\mu_G(B^G_{\rho}(r_{n_i}))} \leq \frac{\e}{2}\,. \] We conclude that \[ \E\left[\frac{\mu_G\left(\vphantom{\bigoplus}\left\{ x \in B^G_{\rho}(r_{n_i}) : \E\left[d_G(X_0, X_{2n_i})^2 \mid X_0 = x, (G,\rho)\right] \geq (C' \e^{-13}) n_i\right\}\right)}{\mu_G(B^G_{\rho}(r_{n_i}))}\right] \leq \frac{\e}{2} + (1-\Pr[\cE_i]) \leq \e\,,\] completing the proof. \end{proof}

\section{Existence of exceptional times}

 We now present an example showing that one cannot hope to prove \pref{thm:goodtimes} for all times. For ease of notation, throughout this section, if $G$ is a graph, we use $V(G)$ and $E(G)$ for the vertex and edge set of $G$, respectively.

\begin{theorem}\label{thm:exceptional}
There exists a stationary random rooted graph $(G,\rho)$ with the following properties: \begin{enumerate}

\item Almost surely:  For any $x\in V(G)$ and $r>0,$ it holds that $|B^G_{x}(r)|\le O(r^7).$

\item Almost surely: $\displaystyle{\sup_{x\in V(G)}} \deg_G(x)=4$.

\item Let $f : \N \to \N$ be an unbounded, monotone increasing function.
Then
there is a sequence of times $\{t_k\}_{k\in \N}$ so that
 \[\lim_{k \to \infty} \Pr\left({\E}\left[d_G(X_0, X_{t_k})^2 \mid (G,\rho), X_0=\rho\right] \ge  t_k \frac{\log t_k}{f(t_k)} \right) =1\,.
\]
\end{enumerate} \end{theorem}

\begin{remark}
With more effort, it is possible to obtain a similar construction with $|B^G_x(r)| \leq r^{2+o(1)}$.
\end{remark}

The basic idea of the construction is simple:  Let $G_n[n]$ denote the result of taking a $3$-regular expander graph on $n$ vertices
and replacing every edge by a path of length $n$.  Then by construction, the volume growth is at most quadratic,
but after time $\approx n^2 \log n$, the random walk will have gone distance $\approx n^2 (\log n)^2$,
making it slightly superdiffusive.  The technical difficulties lie in converting
this finite family of examples into a stationary random graph.
To accomplish this, we build a tree of such graphs (see \pref{fig:hk}), with the
sizes decreasing rapidly down the tree, and with buffers between the levels
to enforce polynomial volume growth.

\subsection{Trees of graphs}

\begin{figure}
\begin{center}
 %\centering
\subfigure[Illustration of $G$ over many copies of $H$ in $\newgraph(G;H)$]{
\includegraphics[width=8.12cm]{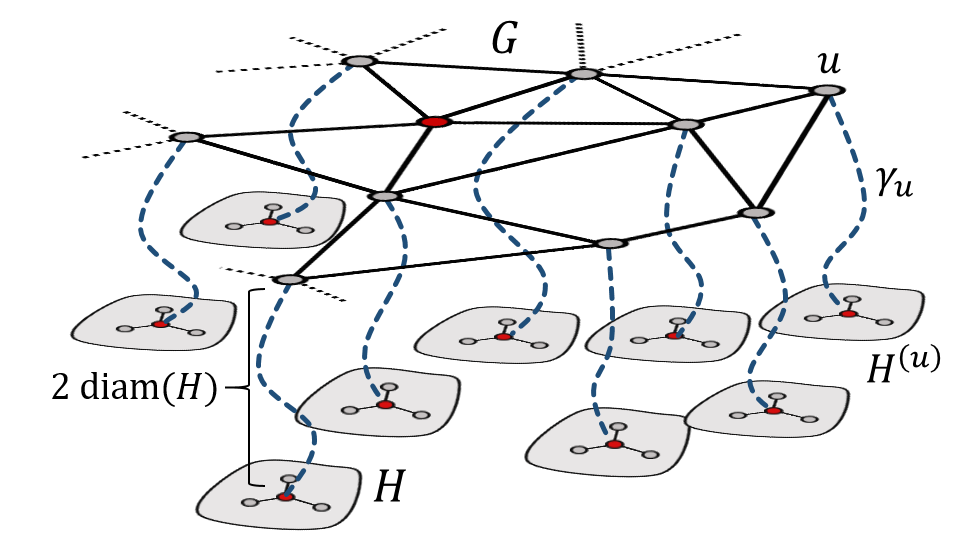}\label{fig:GoverH}
}\hspace{0.11cm} %\hspace*{\fill}%
\subfigure[A partially drawn copy of $H_k$ inside $H_{k+1}$]{\makebox[8cm]{\includegraphics[width=8.12cm]{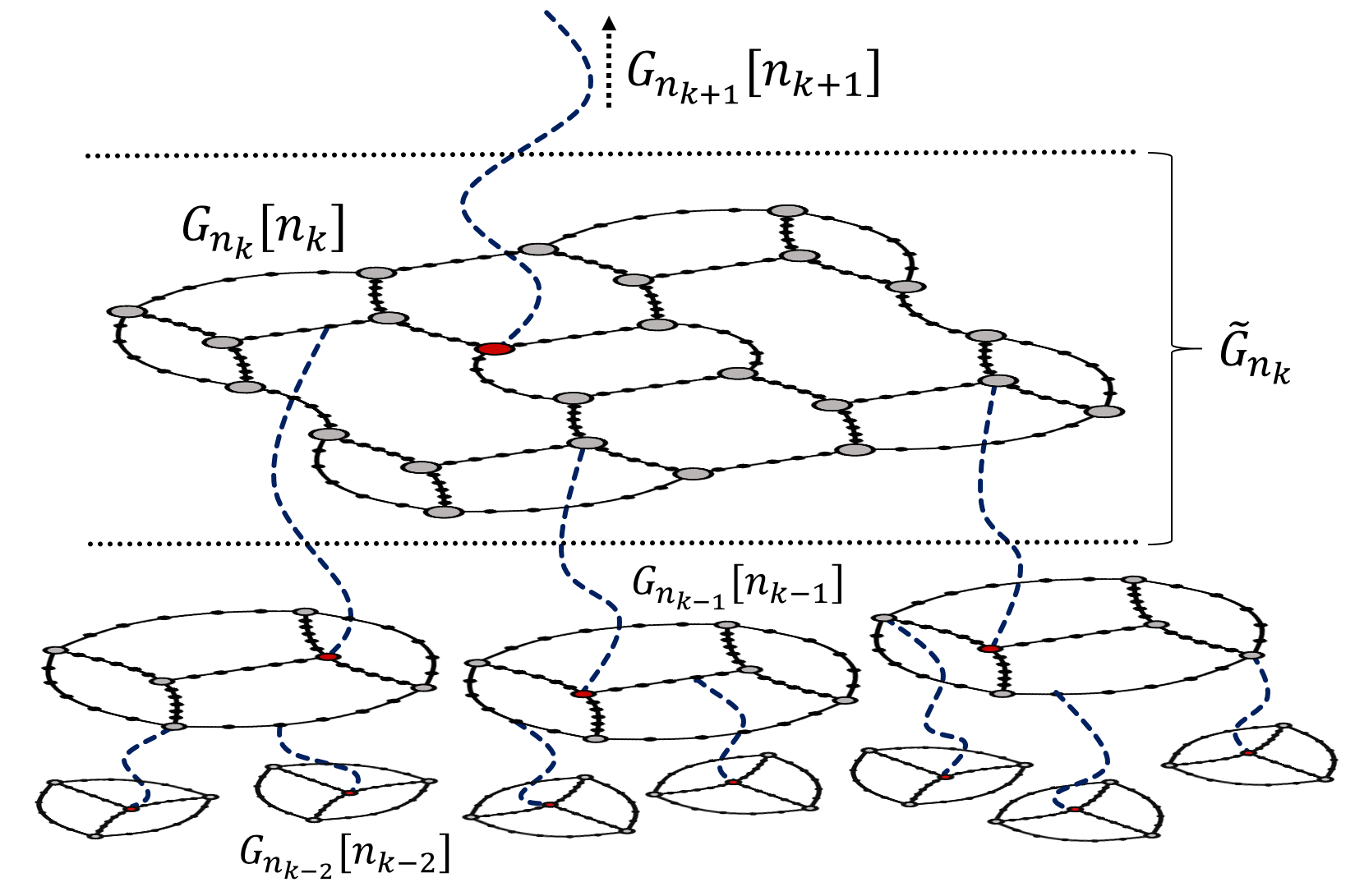}\label{fig:hk}}}
\caption{Trees of graphs}
\end{center}
\end{figure}

We first describe a certain way of constructing graphs from other graphs and provide some preliminary estimates on the properties of the construction.
In this section, we will deal primarily with rooted graphs.  For a graph $G$, we use $\rho_G \in V(G)$ to denote its root.

\medskip \paragraph{A tree of $H$'s under $G$}

Consider two rooted graphs $H$ and $G$.  Construct a new rooted graph $\newgraph=\newgraph(G;H)$ as follows.
Take $|V(G)|-1$ disjoint copies of $H$: $\{H^{(u)} : u \in V(G) \setminus \{\rho_G\}\}$.
Let $\rho_H^{(u)}$ be the copy of $\rho_H$ in $H^{(u)}$.
Let $\{ \gamma_u : u \in V(G) \setminus \{\rho_G\}\}$ be a collection of edge-disjoint paths of length $2\, \diam(H)$
 where $\gamma_u$ connects $u \in V(G) \setminus \{\rho_G\}$ to $\rho_H^{(u)}$ in $H^{(u)}$. Define \begin{align*} V(\newgraph) &= V(G) \cup \bigcup_{u \in V(G) \setminus \{\rho_G\}} \left(V(H^{(u)}) \cup V(\gamma_u)\right)\,, \\ E(\newgraph) &= E(G) \cup \bigcup_{u \in V(G) \setminus \{\rho_G\}} \left(E(H^{(u)}) \cup E(\gamma_u)\right)\,. \end{align*} There is a natural identification $V(G) \subseteq V(\newgraph)$ and we define the root $\rho_{\newgraph}=\rho_G$ of $\newgraph$. We refer to the paths $\{\gamma_u\}$ as {\em tails.}
See \pref{fig:GoverH}.

\medskip For a graph $G$, let us use $\Delta_G = \max_{v \in V(G)} \deg_G(v)$ to denote its maximum degree.

\begin{lemma} \label{lem:degrees} For any rooted graphs $H$ and $G$ and $h \geq 1$, if $\newgraph = \newgraph(G;H)$, then $\deg_{\newgraph}(\rho_{\newgraph}) = \deg_G(\rho_G)$, and
\begin{align*}
 \Delta_{\newgraph} &\leq \max\{\Delta_H, \Delta_G+1,\deg_H(\rho_H)+1\}\,. \\
 |V(\newgraph)| &= |V(G)| + \left(\vphantom{\bigoplus} |V(G)|-1\right) \left(\vphantom{\bigoplus}|V(H)|+2\,\diam(H)-2\right)\,, \\ \diam(\newgraph) &\leq \diam(G) + 6\,\diam(H)\,. \end{align*} \end{lemma}

\paragraph{Graph subdivision}

For a parameter $L \in \N$, we define a graph $ G[L]$ as the one which arises from $G$ by subdividing every edge in $E(G)$ into a path of length $L$. If $G$ has root $\rho_G$, then under the natural identification $V(G) \subseteq V(G[L])$, we set $\rho_{G[L]}=\rho_G$.
Note that: \begin{equation}\label{eq:subsize} |V(G[L])| = |V(G)| + |E(G)|\cdot (L-1)\,. \end{equation}

\subsection{Stretched expanders and the rate of escape}

Let $\{G_n : n \in 2\N\}$ denote a family of $3$-regular, $n$-vertex non-bipartite expander graphs. For each such $n$, we distinguish an arbitrary root $\rho_{G_n} \in V(G_n)$.  We use $\tmix(G)$ to denote the (total variation) mixing time of a graph $G$.

\begin{fact} There is a constant $C > 0$ such that $\tmix(G_n) \leq C \log n$ for all $n \in 2\N$.
\end{fact}

Since $G_n$ is $3$-regular, a fixed vertex is further than $\frac13 \log n$ from
all but $o(n)$ vertices in $G_n$.
Combining this with the preceding fact yields the following.

\begin{lemma}
There is a constant $c \geq 1$ such that the following holds for every $\e \in (0,1)$
and $n \geq 1/\e^3$.
If $\{X_t\}$ is the random walk in $G_n$ and
$t \geq (c/\e) \tmix(G_n)$, then
\begin{equation}\label{eq:speed}
\min_{x \in V(G_n)} \Pr\left[d_{G_n}(X_0,X_t) > \tfrac13 \log n \mid X_0=x\right] \geq 1-\e. \end{equation}
\end{lemma}

For a graph $H$, define
\begin{align*}
\hit(H) &= \max_{x,y \in V(H)} \E\left[\min \left\{ t \geq 0 : Y_t = y\right\} \mid Y_0=x\right]\,,
\end{align*}
where $\{Y_t\}$ is the random walk on $H$.

One has the following basic estimate (see, e.g., \cite[Ch. 10]{LPW09}):
\begin{equation}\label{eq:hit} \hit(H) \leq 2 \Delta_H |V(H)|^2\,. \end{equation}
Let $\pi_G$ denote the stationary measure of the random walk on a graph $G$.

\begin{lemma}\label{lem:nozero}
There is a constant $c' > 0$ such that the following holds
for all $\e \in (0,1)$ and $n \geq 1/\e^3$.
Let $\{X_t\}$ denote the random walk on $\newgraph=\newgraph(G_n[L];H)$,
with $X_0$ chosen
according to the stationary measure $\pi_{\newgraph}$.

Assume that $L \geq  \diam(H)$ and $\log n \geq 48$.
Then for all \[t \in \left[\frac{c'}{\e^2} \Delta_H |V(H)|^2 L^2 \tmix(G_n),\frac{\e L^2 n}{c'}\right]\,,\] it holds that
\[
\Pr\left(\E\left[d_{\newgraph}(X_t, X_0)^2 \1_{\left\{\rho_{\newgraph} \notin \left\{X_0, X_1, \ldots, X_t\right\}\vphantom{\bigoplus}\right\}} \mid X_0\right] \geq
\frac{(L \log n)^2}{72} \right)
 \geq 1-\e\,.
\]
\end{lemma}

\begin{proof}
Let $\tmix =  (3c/\e)\tmix(G_n)$.
Let $\hat{\newgraph}$ denote the graph $\newgraph$, but where a path $\hat \gamma$ of length $2\,\diam(H)$ is added between $\rho_{\newgraph}$ and a new copy $H^{(\rho_G)}$ of $H$ (so that now all vertices of $G_n[L]$ have a copy of $H$ attached).

Consider the random walk $\{Z_k\}$ on $\hat{\newgraph}$. Let $\tau_1 < \tau_2 < \cdots$ be the sequence of times at which $Z_{\tau_j} \in V(G_n)$. Let $K = \max \{ j : \tau_j < t \}$ (and let $K=0$ if no such $j$ exists).
Observe that, conditioned on the sequence $\{\tau_j\}$, the process $\{Z_{\tau_1}, Z_{\tau_2}, \ldots\}$ has the law of random walk on $G_n$, therefore using \eqref{eq:speed} yields \begin{align}\label{eq:ktmix}
\Pr\left[\{K > 0\} \wedge d_{G_n}(Z_{\tau_K}, Z_{\tau_1}) \geq \tfrac13 \log n \mid Z_0\right] \geq (1-\tfrac13 \e) \Pr[K \geq \tmix \mid Z_0]\,.
\end{align}

The waiting periods $\{\tau_{j+1}-\tau_j : j=1,2,\ldots\}$ are i.i.d., and we have the estimates
\begin{align*}
\E[\tau_1 \mid Z_0], \E[\tau_{j+1}-\tau_j] &\stackrel{{\eqref{eq:hit}}}{\leq} 3 \Delta_H L^2 |V(H)|^2\,.
\end{align*}
Hence,
\begin{align*} \Pr[K < \tmix \mid Z_0]  \leq \Pr[\tau_{\lfloor \tmix\rfloor} > t \mid Z_0] \leq \frac{\E[\tau_{\lfloor \tmix\rfloor} \mid Z_0]}{t} \leq
\frac{3 \Delta_H L^2 |V(H)|^2 \tmix}{t}\,.
\end{align*}

Combined with \eqref{eq:ktmix}, this shows for $t \geq (9/\e) \Delta_H L^2 |V(H)|^2 \tmix$,
\begin{equation}\label{eq:almost}
\Pr\left[\{K > 0\} \wedge d_{G_n}(Z_{\tau_1}, Z_{\tau_K}) \geq \tfrac13 \log n \mid Z_0\right] \geq (1-\tfrac13 \e)^2 \geq 1-\tfrac23 \e\,.
\end{equation}
Now, note that as long as $K > 0$ and $\rho_{\newgraph} \notin \{Z_0, \ldots, Z_t\}$, we can couple $\{Z_0, \ldots, Z_t\}$ with the random walk $\{X_0, \ldots, X_t\}$ on $\newgraph$.
For any $T \geq 1$,
\begin{align}
\Pr\left[\rho_{\newgraph} \in \{Z_0, Z_1, \ldots, Z_t\}\right] &\leq \Pr[K > T] + \Pr[\rho_{\newgraph} \in \{Z_{\tau_1}, Z_{\tau_2}, \ldots, Z_{\tau_K}\} \mid K \leq T] \nonumber
\\ &\leq \Pr[K > T] + \frac{(T+1)}{n}\,, \label{eq:Tbnd}
\end{align}
where the last inequality follows because $\{Z_{\tau_1}, \ldots, Z_{\tau_K}\}$ is a stationary walk on $G_n$, conditioned on $\{\tau_j\}$, and because $G_n$ is regular.

We now require a basic estimate on $\tau_{j+1}-\tau_j$.  Let $Y$ denote the amount of time
needed for a random walk on $\Z$, started at the origin, to hit the set $\{-L,L\}$.  Then
$\tau_{j+1}-\tau_j$ stochastically dominates $Y$, and we have the standard identities (see, e.g., \cite{Moon73}):
\begin{align*}
\E[Y] &= L^2 \\
\Var(Y) &= \frac{2(L^4-L^2)}{3}\,.
\end{align*}
Let $\{Y_j\}$ be i.i.d. copies of $Y$, and use Chebyshev's inequality to obtain:
\[
\Pr\left(\tau_{m+1} < \tfrac12 m L^2\right) \leq \Pr\left(Y_1+\cdots+Y_{m} < m \E[Y] - \tfrac12 m L^2\right) \leq \frac{8}{3m}\,.
\]
For $T = \lceil \frac{2t}{L^2}\rceil$, this yields
\[
\Pr \left(K > T\right) \leq \Pr\left(\tau_{T+1} < t\right) \leq \frac{8}{3T} \leq \frac{2 L^2}{t}\,.
\]
Plugging this into \eqref{eq:Tbnd} gives
\[
p \seteq \Pr[\rho_{\newgraph} \in \{Z_0, Z_1, \ldots, Z_t\}] \leq \frac{2L^2}{t} + \frac{t}{L^2 n}\,.
\]
Note that $p \leq \e/6$ as long as $t \in [\frac{24 L^2}{\e} , \frac{\e L^2 n}{12}]$.

Using this in conjunction with \eqref{eq:almost}, we arrive at
\begin{align}\nonumber
\Pr\left(\Pr\left[\{K > 0 \} \wedge \rho_{\newgraph} \notin \{X_0, X_1, \ldots,X_t\} \wedge d_{G_n}(X_{\tau_1},X_{\tau_K}) \geq \frac{\log n}{3} \mid X_0 \right] \geq \frac12\right) &\geq (1-\tfrac23 \e) - 2p \\ &\geq 1-\e\,.\label{eq:almostX}
\end{align}
When $K > 0$, the triangle inequality gives us
\[
d_{\newgraph}(X_0, X_t) \geq L \cdot d_{G_n}(X_{\tau_1},X_{\tau_K}) - 2(L+3\,\diam(H))\,.
\]
Combining this with \eqref{eq:almostX}, along with the assumptions that $L \geq \diam(H)$ and $\log n \geq 48$ yields
\begin{align*}
\Pr\left(\E\left[d_{G_n}(X_{0},X_{t})^2 \1_{\left\{\rho_{\newgraph} \notin \left\{X_0, X_1, \ldots, X_t\right\}\vphantom{\bigoplus}\right\}}\mid X_0 \right]
\geq
\frac12 \frac{(L \log n)^2}{36} \right) \geq 1-\e,.
\end{align*}
completing the proof.
\end{proof}

\subsection{The recursive construction}

Observe that there is a constant $C > 0$ such that for $n \geq 4$, \begin{equation}\label{eq:expdiam} \diam(G_n) \leq C \log n\,. \end{equation}

Let us denote $n_0 = 10$ and suppose that $n_k \geq 2 n_{k-1}^2$ for $k \geq 1$. We define an inductive sequence of rooted graphs $\{H_k\}$ as follows:  $H_0$ is the graph
consisting of a single vertex, and for $k \geq 1$, \begin{align*} H_{k} &= \newgraph(G_{n_k}[n_k]; H_{k-1})\,. \end{align*}
Refer to \pref{fig:hk} for a depiction.

We begin by consulting \pref{lem:degrees} for the following estimates.  Using \eqref{eq:expdiam}, we have for $k \geq 1$: \begin{align*} \diam(H_k) &\leq \diam(G_{n_k}[n_k]) + 6\, \diam(H_{k-1}) \leq C n_k \log n_k + 6\, \diam(H_{k-1})\,.
\end{align*} Thus one can easily verify by induction that \begin{equation}\label{eq:diambnd} \diam(H_k) \leq 2 C n_k \log n_k\,. \end{equation} Moreover, \[ |V(H_k)| \leq 2 n_k^2 \left(|V(H_{k-1})|+2\,\diam(H_{k-1})\right) \leq 2 n_k^2 (|V(H_{k-1})| + 4 C n_{k-1} \log n_{k-1})\,, \] and one  verifies by induction that for $k \geq 1$, \begin{equation}\label{eq:sizebnd} n_k^2 \leq |V(H_k)| \leq 2C n_k^4\,. \end{equation} From \pref{lem:degrees}, the following bound on the vertex degrees is immediate: \begin{equation} \label{eq:degrees} \sup_{k \geq 1} \Delta_{H_k} \leq 4\,. \end{equation}

\paragraph{Levels of vertices}
The graph $H_k$ consists of a copy of $G_{n_k}[n_k]$ connected to $|V(G_{n_k}[n_k])|-1$ copies of $H_{k-1}$ via tails of length $2\,\diam(H_{k-1})$. Each such copy of $H_{k-1}$ contains a copy of $G_{n_{k-1}}[n_{k-1}]$ that is connected to $|V(G_{n_{k-1}}[n_{k-1}])|-1$ copies of $H_{k-2}$ via tails of length $2\,\diam(H_{k-2})$, and so on. If a tail connects $G_{n_{\ell}}[n_{\ell}]$ to $H_{\ell-1}$, we refer to it as a {\em level-$\ell$ tail.}

Naturally, we can think of every vertex as occurring in either in a copy of $G_{n_{\ell}}$ for some $\ell \leq k$, or in a tail between $G_{n_{\ell}}[n_{\ell}]$
and a copy of $H_{\ell-1}$.
Let $\cT^k_{\ell} \subseteq V(H_k)$ denote the set of internal vertices in level-$\ell$ tails.  Let $V_{\ell}^k \subseteq V(H_k)$ denote the set of vertices occurring in some copy of $G_{n_{\ell}}[n_{\ell}]$. Note that the sets $\{ V_{\ell}^k, \cT^k_{\ell} : \ell = 1,2,\ldots,k\}$ form a partition of $V(H_k)$.

Finally, we use the notation $\tilde G_{n_{\ell}}$ for the graph $G_{n_{\ell}}[n_{\ell}]$, together with the tail of length $\diam(H_{\ell})-1$ attached to $\rho_{G_{n_\ell}[n_{\ell}]}$, and the disjoint tails of length $\diam(H_{\ell-1})+1$ attached to all vertices of $V(G_{n_{\ell}}[n_{\ell}]) \setminus \{\rho_{G_{n_\ell}[n_{\ell}]}\}$. Observe that $V(H_k)$ partitions into a {\em disjoint union} of copies of $\tilde G_{n_{\ell}}$ with $\ell \in \{1,2,\ldots,k\}$. Accordingly, we can write $\level(x)$ for the index $\ell$ such that $x$ is in a copy of $\tilde G_{n_{\ell}}$, and $\nbors_x$ for the subgraph corresponding to $x$'s copy of $\tilde G_{n_{\level(x)}}$. We now observe the main point of the tails.

\begin{lemma}\label{lem:diffcopies} Suppose that $x,y \in V(H_k)$.  If  $\tilde G_x \neq \tilde G_y$, then either $x$ and $y$ lie on a common tail, or $d_{H_k}(x,y) \geq \diam\left(H_{\min(\level(x),\level(y))}\right)$.
\end{lemma}

\begin{lemma}\label{lem:Gtilde} For any $n \in 2\N$, $v \in V(\tilde G_{n})$, and $r \geq 1$, it holds that \[ |B_v^{\tilde G_n}(r)| \leq 3 r^3\,. \] \end{lemma}

\begin{proof} If $r \leq n$, then \[|B_v^{\tilde G_n}(r)| \leq r |B_v^{G_n[n]}(r)| \leq r^2 \Delta_{G_n} \leq 3r^2\,.\] Otherwise, $|B_v^{\tilde G_n}(r)| \leq r |V(G_n[n])| \leq 2 r n^2 \leq 2 r^3$. \end{proof}

Observe also the basic estimate:  For $k \geq 1$, \begin{equation}\label{eq:easydiam} \diam(H_k) \geq n_k \,\diam(G_{n_k}) \geq 2 \,n_k\,. \end{equation}

\begin{lemma}[Polynomial volume growth] \label{lem:polyvol} For every $k \geq 1$, $r \geq 0$, and $x \in V(H_k)$, it holds that \[ |B_x^{H_k}(r)| \leq 6C r^7\,. \] \end{lemma}

\begin{proof} Consider $n_{\ell-1} \leq r < n_{\ell}$. The main idea is that if $j > \ell$, then from \pref{lem:diffcopies}, we know the ball $B_x^{H_k}(r)$ cannot intersect both the top and bottom half of a level-$j$ tail $\gamma$ unless $B_x^{H_k}(r) \subseteq V(\gamma)$, because the length of $\gamma$ is at least $2 \,\diam(H_{\ell}) \geq 4 n_{\ell}$.

More precisely, from \pref{lem:diffcopies} we know that one of the following cases occurs:
\begin{enumerate}
\item The ball $B_x^{H_k}(r)$ is completely contained in some tail. In this case, clearly $|B_x^{H_k}(r)| \leq 2r$.
\item The ball $B_x^{H_k}(r)$ is contained in a copy of $\tilde G_{n_j}$ for an index $j > \ell$.  In this case, $|B_x^{H_k}(r)| \leq 3r^3$ from \pref{lem:Gtilde}.
\item It holds that $\max \left\{ \level(v) : v \in B_x^{H_k}(r)\right\}\leq \ell$.
In this case, use \pref{lem:Gtilde} to write
\[ |B_x^{H_k}(r)| \leq \left(\max_{v \in V(\tilde G_{n_{\ell}})} \left|B_v^{\tilde G_{n_{\ell}}}(r)\right|\right) |V(H_{\ell-1})| \leq 3r^3 |V(H_{\ell-1})|
\stackrel{\eqref{eq:sizebnd}}{\leq} 6 C r^3  n_{\ell-1}^4 \leq 6Cr^7
\,.\qedhere \]
\end{enumerate}\end{proof}

Let $\{X_t\}$ denote the random walk on $H_k$ where $X_0$ has law $\pi_{H_k}$.

\begin{lemma}[Speed of the random walk]
\label{lem:finitespeed}
There is a constant $c > 0$ such that the following
holds:  For all $\e > 0$ and $\ell$ sufficiently large (with respect to $\e$),
if $k \geq \ell$ and
\[
t \in \left[\frac{c}{\e^2} n_{\ell}^2 (\log n_{\ell}) |V(H_{\ell-1})|^2,\frac{\e n_{\ell}^3}{c}\right]\,,
\]
then
\[
\Pr\left(\E\left[d_{H_k}(X_0, X_t)^2 \mid X_0\right] \geq \frac{\left(n_{\ell} \log n_{\ell}\right)^2}{72}\right)  \geq 1-\e\,.
\]
\end{lemma}

\begin{proof}
For $1 \leq \ell \leq k$, let $p_{\ell,k}$ denote the probability
that a vertex $v \in V(H_k)$ chosen uniformly at random
{\em does not} fall in some copy of $H_{\ell}$.
First, we use \eqref{eq:sizebnd} and \eqref{eq:diambnd} to bound
\begin{equation}\label{eq:pk}
p_{k-1,k} \leq \frac{|V(G_{n_k}[n_k])|\cdot 2\,\diam(H_{k-1})}{\left(|V(G_{n_k}[n_k])|-1\right)|V(H_{k-1})|}
\leq \frac{O(\log n_{k-1})}{n_{k-1}}\,,
\end{equation}
This yields
\[
p_{\ell,k} \leq \sum_{j=\ell}^{k-1} p_{j,j+1} \leq O(1)  \sum_{j=\ell}^{k+1} \frac{\log n_j}{n_j}\,.
\]
Observe that since $\Delta_{H_k} \leq 4$, the probability that a vertex
chosen from the stationary measure does not fall in some copy of $H_{\ell}$ is bounded by $4 p_{\ell,k}$.
Recall that the sequence $\{n_j\}$ is increasing rapidly:  $n_{j+1} \geq 2 n_j^2$.
Let $\ell$ be chosen large enough so that $4 p_{\ell,k} < \e$.

Let $\cE(\ell)$ denote the event that $X_0$ lies in a copy of $H^*_{\ell}$
of $H_{\ell}$.
We have $\Pr[\cE(\ell)] \geq 1-\e$.
Moreover, conditioned on $\cE(\ell)$, if the random walk $\{X_0, X_1, \ldots, X_t\}$ avoids
the root $\rho_{H^*_{\ell}}$, then it can be coupled to a stationary random walk on $H_{\ell}$.
Now applying \pref{lem:nozero} with $\newgraph=H_{\ell}$ yields the desired result.
\end{proof}

\subsection{Convergence to a stationary random graph}

Let $\rho_k \in V(H_k)$ be chosen according
to the stationary measure $\pi_{H_k}$,
and let $\mu_k$ be the law of the random rooted graph $(H_k, \rho_k)$.

 \begin{lemma}\label{lem:local-limit}
The measure $\mu \seteq \lim_{k \to \infty} \mu_k$ exists in the local weak topology.
Moreover, if $(G,\rho)$ has the law of $\mu$, then $(G,\rho)$ is a stationary
 random graph such that, almost surely, $\sup_{x \in V(G)} \deg_G(x) \leq 4$, and $|B_{\rho}^G(r)| \leq 6C r^7$ for all $r \geq 1$. \end{lemma}

\begin{proof} Assuming that the limit exists, the latter assertions follow from \eqref{eq:degrees} and \pref{lem:polyvol}.

By definition of the local weak topology, to prove convergence of the measures $\mu_k$, it suffices to show that for every $r > 0$, the measures $\mu_{k,r}$ converge, where $\mu_{k,r}$ is the law of $B_{\rho_k}^{H_k}(r)$. A standard application of Kolmogorov's extension theorem then proves the existence of the limit $\mu$. For more details, see \cite{ben-schramm}.

Let $\cE$ denote the event that $\rho_k$ lies in a copy of $H_{k-1}$. Observe that (recall \eqref{eq:pk}):
\[ \Pr[\neg \cE] \leq \frac{\Delta_{H_k} |V(G_{n_k}[n_k])|\cdot 2\,\diam(H_{k-1})}{\left(|V(G_{n_k}[n_k])|-1\right)|V(H_{k-1})|} \leq O\left(\frac{\log n_{k-1}}{n_{k-1}}\right)\,. \]

Suppose that $\cE$ occurs, and let $H^*_{k-1}$ denote the copy of $H_{k-1}$ in $H_k$
containing $\rho_k$.  In this case, we can couple $\rho_k$ and $\rho_{k-1} \in V(H_{k-1})$ in the obvious way.
Note furthermore that $(B_{\rho_k}^{H_k}(r),\rho_k)$ and $(B_{\rho_{k-1}}^{H_{k-1}}(r),\rho_{k-1})$ are coupled
(under the natural isomorphism) as long as $d_{H_{k-1}}(\rho_{k-1}, \rho_{H_{k-1}}) > r$. This yields \[ \Pr\left[d_{H_{k-1}}(\rho_{k-1}, \rho_{H_{k-1}}) \leq r\right] \leq \frac{|B_{\rho_{H_{k-1}}}(r)|}{|V(H_{k-1})|} \leq \frac{4^r}{n_{k-1}^2}\,, \]
since $\Delta_{H_{k-1}} \leq 4$. We conclude that, for fixed $r > 0$, it holds that \[ d_{TV}\left(\mu_{k-1,r},\mu_{k,r}\right) \leq O\left(\frac{\log n_{k-1}}{n_{k-1}}\right)\,. \]
Since $\frac{\log n_{k}}{n_{k}}$ is summable, this yields
the desired convergence as $k \to \infty$.
\end{proof}

We are ready to complete the proof of \pref{thm:exceptional}.

\begin{proof}[Proof of \pref{thm:exceptional}]
Let $(G,\rho)$ be the limit of $(H_k,\rho_k)$ constructed in \pref{lem:local-limit}.
Properties (1) and (2) are satisfied by the statement of the lemma.
Let $\{\e_\ell\}$ denote a sequence with $\e_{\ell} \to 0$ as $\ell \to \infty$
and such that \pref{lem:finitespeed} applies to $\e=\e_{\ell}$ for $k \geq \ell$.

The third property follows from \pref{lem:finitespeed}
by choosing the sequence $\{n_{\ell}\}$ to grow fast enough so that
\[
|V(H_{\ell-1})|^2 \stackrel{\eqref{eq:sizebnd}}{\leq} O(n_{\ell-1}^8) \leq \frac{o(\log n_{\ell})}{\e_{\ell}^2 f(n_{\ell})}\,.
\]
as $\ell \to \infty$.
\end{proof}

\bibliographystyle{alpha} \bibliography{diffusive,mt}

\begin{thebibliography}{BDCKY15}

\bibitem[AL07]{aldous-lyons}
David Aldous and Russell Lyons.
\newblock Processes on unimodular random networks.
\newblock {\em Electron. J. Probab.}, 12:no. 54, 1454--1508, 2007.

\bibitem[Bal92]{Ball92}
K.~Ball.
\newblock Markov chains, {R}iesz transforms and {L}ipschitz maps.
\newblock {\em Geom. Funct. Anal.}, 2(2):137--172, 1992.

\bibitem[BC12]{BC12}
Itai Benjamini and Nicolas Curien.
\newblock Ergodic theory on stationary random graphs.
\newblock {\em Electron. J. Probab.}, 17:no. 93, 20, 2012.

\bibitem[BDCKY15]{benjamini}
Itai Benjamini, Hugo Duminil-Copin, Gady Kozma, and Ariel Yadin.
\newblock Disorder, entropy and harmonic functions.
\newblock {\em Ann. Probab.}, 43(5):2332--2373, 2015.

\bibitem[Bow95]{bowditch95}
B.~H. Bowditch.
\newblock A short proof that a subquadratic isoperimetric inequality implies a
  linear one.
\newblock {\em Michigan Math. J.}, 42(1):103--107, 1995.

\bibitem[BP89]{BP89}
Martin~T. Barlow and Edwin~A. Perkins.
\newblock Symmetric {M}arkov chains in {$\mathbb{Z}^d$}: how fast can they
  move?
\newblock {\em Probab. Theory Related Fields}, 82(1):95--108, 1989.

\bibitem[BS01]{ben-schramm}
Itai Benjamini and Oded Schramm.
\newblock Recurrence of distributional limits of finite planar graphs.
\newblock {\em Electron. J. Probab.}, 6:no. 23, 13 pp. (electronic), 2001.

\bibitem[Car85]{Carne85}
Thomas~Keith Carne.
\newblock A transmutation formula for {M}arkov chains.
\newblock {\em Bull. Sci. Math. (2)}, 109(4):399--405, 1985.

\bibitem[CKR01]{CKR01}
Gruia Calinescu, Howard Karloff, and Yuval Rabani.
\newblock Approximation algorithms for the 0-extension problem.
\newblock In {\em Proceedings of the 12th Annual ACM-SIAM Symposium on Discrete
  Algorithms}, pages 8--16, Philadelphia, PA, 2001.

\bibitem[DLP13]{DLP13}
Jian Ding, James~R. Lee, and Yuval Peres.
\newblock Markov type and threshold embeddings.
\newblock {\em Geom. Funct. Anal.}, 23(4):1207--1229, 2013.

\bibitem[HSC93]{HS93}
W.~Hebisch and L.~Saloff-Coste.
\newblock Gaussian estimates for {M}arkov chains and random walks on groups.
\newblock {\em Ann. Probab.}, 21(2):673--709, 1993.

\bibitem[Kes59]{Kesten59}
Harry Kesten.
\newblock Symmetric random walks on groups.
\newblock {\em Trans. Amer. Math. Soc.}, 92:336--354, 1959.

\bibitem[Kes86]{Kesten86}
Harry Kesten.
\newblock Subdiffusive behavior of random walk on a random cluster.
\newblock {\em Ann. Inst. H. Poincar\'e Probab. Statist.}, 22(4):425--487,
  1986.

\bibitem[Kle10]{Kleiner10}
Bruce Kleiner.
\newblock A new proof of {G}romov's theorem on groups of polynomial growth.
\newblock {\em J. Amer. Math. Soc.}, 23(3):815--829, 2010.

\bibitem[KLMN05]{KLMN05}
R.~Krauthgamer, J.~R. Lee, M.~Mendel, and A.~Naor.
\newblock Measured descent: {A} new embedding method for finite metrics.
\newblock {\em Geom. Funct. Anal.}, 15(4):839--858, 2005.

\bibitem[LP16]{LP:book}
Russell Lyons and Yuval Peres.
\newblock {\em Probability on Trees and Networks}.
\newblock Cambridge University Press, 2016.
\newblock Available at \url{http://pages.iu.edu/~rdlyons/}.

\bibitem[LPW09]{LPW09}
David~A. Levin, Yuval Peres, and Elizabeth~L. Wilmer.
\newblock {\em Markov chains and mixing times}.
\newblock American Mathematical Society, Providence, RI, 2009.
\newblock With a chapter by James G. Propp and David B. Wilson.

\bibitem[LZ94]{LZ94}
T.~J. Lyons and T.~S. Zhang.
\newblock Decomposition of {D}irichlet processes and its application.
\newblock {\em Ann. Probab.}, 22(1):494--524, 1994.

\bibitem[Moo73]{Moon73}
J.~W. Moon.
\newblock Random walks on random trees.
\newblock {\em J. Austral. Math. Soc.}, 15:42--53, 1973.

\bibitem[NPSS06]{NPSS06}
Assaf Naor, Yuval Peres, Oded Schramm, and Scott Sheffield.
\newblock Markov chains in smooth {B}anach spaces and {G}romov-hyperbolic
  metric spaces.
\newblock {\em Duke Math. J.}, 134(1):165--197, 2006.

\bibitem[Pin94]{azuma}
Iosif Pinelis.
\newblock Optimum bounds for the distributions of martingales in {B}anach
  spaces.
\newblock {\em Ann. Probab.}, 22(4):1679--1706, 1994.

\bibitem[ST10]{ST10}
Yehuda Shalom and Terence Tao.
\newblock A finitary version of {G}romov's polynomial growth theorem.
\newblock {\em Geom. Funct. Anal.}, 20(6):1502--1547, 2010.

\bibitem[Var85]{Var85}
Nicholas~Th. Varopoulos.
\newblock Long range estimates for {M}arkov chains.
\newblock {\em Bull. Sci. Math. (2)}, 109(3):225--252, 1985.

\bibitem[Woe00]{Woess00}
Wolfgang Woess.
\newblock {\em Random walks on infinite graphs and groups}, volume 138 of {\em
  Cambridge Tracts in Mathematics}.
\newblock Cambridge University Press, Cambridge, 2000.

\end{thebibliography}


\begin{thebibliography}{BDCKY15}

\bibitem[Bal92]{Ball92}
K.~Ball.
\newblock Markov chains, {R}iesz transforms and {L}ipschitz maps.
\newblock {\em Geom. Funct. Anal.}, 2(2):137--172, 1992.

\bibitem[BDCKY15]{benjamini}
Itai Benjamini, Hugo Duminil-Copin, Gady Kozma, and Ariel Yadin.
\newblock Disorder, entropy and harmonic functions.
\newblock {\em Ann. Probab.}, 43(5):2332--2373, 2015.

\bibitem[Bow95]{bowditch95}
B.~H. Bowditch.
\newblock A short proof that a subquadratic isoperimetric inequality implies a
  linear one.
\newblock {\em Michigan Math. J.}, 42(1):103--107, 1995.

\bibitem[CKR01]{CKR01}
Gruia Calinescu, Howard Karloff, and Yuval Rabani.
\newblock Approximation algorithms for the 0-extension problem.
\newblock In {\em Proceedings of the 12th Annual ACM-SIAM Symposium on Discrete
  Algorithms}, pages 8--16, Philadelphia, PA, 2001.

\bibitem[DLP13]{DLP13}
Jian Ding, James~R. Lee, and Yuval Peres.
\newblock Markov type and threshold embeddings.
\newblock {\em Geom. Funct. Anal.}, 23(4):1207--1229, 2013.

\bibitem[Kes59]{Kesten59}
Harry Kesten.
\newblock Symmetric random walks on groups.
\newblock {\em Trans. Amer. Math. Soc.}, 92:336--354, 1959.

\bibitem[KLMN05]{KLMN05}
R.~Krauthgamer, J.~R. Lee, M.~Mendel, and A.~Naor.
\newblock Measured descent: a new embedding method for finite metrics.
\newblock {\em Geom. Funct. Anal.}, 15(4):839--858, 2005.

\bibitem[NPSS06]{NPSS06}
Assaf Naor, Yuval Peres, Oded Schramm, and Scott Sheffield.
\newblock Markov chains in smooth {B}anach spaces and {G}romov-hyperbolic
  metric spaces.
\newblock {\em Duke Math. J.}, 134(1):165--197, 2006.

\bibitem[Pin94]{azuma}
Iosif Pinelis.
\newblock Optimum bounds for the distributions of martingales in {B}anach
  spaces.
\newblock {\em Ann. Probab.}, 22(4):1679--1706, 1994.

\bibitem[Woe00]{Woess00}
Wolfgang Woess.
\newblock {\em Random walks on infinite graphs and groups}, volume 138 of {\em
  Cambridge Tracts in Mathematics}.
\newblock Cambridge University Press, Cambridge, 2000.

\end{thebibliography}


\begin{thebibliography}{BDCKY15}

\bibitem[AL07]{aldous-lyons}
David Aldous and Russell Lyons.
\newblock Processes on unimodular random networks.
\newblock {\em Electron. J. Probab.}, 12:no. 54, 1454--1508, 2007.

\bibitem[BDCKY15]{benjamini}
Itai Benjamini, Hugo Duminil-Copin, Gady Kozma, and Ariel Yadin.
\newblock Disorder, entropy and harmonic functions.
\newblock {\em Ann. Probab.}, 43(5):2332--2373, 2015.

\bibitem[BS01]{ben-schramm}
Itai Benjamini and Oded Schramm.
\newblock Recurrence of distributional limits of finite planar graphs.
\newblock {\em Electron. J. Probab.}, 6:no. 23, 13 pp. (electronic), 2001.

\bibitem[CKR01]{CKR01}
Gruia Calinescu, Howard Karloff, and Yuval Rabani.
\newblock Approximation algorithms for the 0-extension problem.
\newblock In {\em Proceedings of the 12th Annual ACM-SIAM Symposium on Discrete
  Algorithms}, pages 8--16, Philadelphia, PA, 2001.

\bibitem[DLP13]{DLP13}
Jian Ding, James~R. Lee, and Yuval Peres.
\newblock Markov type and threshold embeddings.
\newblock {\em Geom. Funct. Anal.}, 23(4):1207--1229, 2013.

\bibitem[KLMN05]{KLMN05}
R.~Krauthgamer, J.~R. Lee, M.~Mendel, and A.~Naor.
\newblock Measured descent: {A} new embedding method for finite metrics.
\newblock {\em Geom. Funct. Anal.}, 15(4):839--858, 2005.

\bibitem[NPSS06]{NPSS06}
Assaf Naor, Yuval Peres, Oded Schramm, and Scott Sheffield.
\newblock Markov chains in smooth {B}anach spaces and {G}romov-hyperbolic
  metric spaces.
\newblock {\em Duke Math. J.}, 134(1):165--197, 2006.

\bibitem[Pin94]{azuma}
Iosif Pinelis.
\newblock Optimum bounds for the distributions of martingales in {B}anach
  spaces.
\newblock {\em Ann. Probab.}, 22(4):1679--1706, 1994.

\end{thebibliography}


\begin{thebibliography}{BDCKY15}

\bibitem[AL07]{aldous-lyons}
David Aldous and Russell Lyons.
\newblock Processes on unimodular random networks.
\newblock {\em Electron. J. Probab.}, 12:no. 54, 1454--1508, 2007.

\bibitem[Bal92]{Ball92}
K.~Ball.
\newblock Markov chains, {R}iesz transforms and {L}ipschitz maps.
\newblock {\em Geom. Funct. Anal.}, 2(2):137--172, 1992.

\bibitem[BDCKY15]{benjamini}
Itai Benjamini, Hugo Duminil-Copin, Gady Kozma, and Ariel Yadin.
\newblock Disorder, entropy and harmonic functions.
\newblock {\em Ann. Probab.}, 43(5):2332--2373, 2015.

\bibitem[Bow95]{bowditch95}
B.~H. Bowditch.
\newblock A short proof that a subquadratic isoperimetric inequality implies a
  linear one.
\newblock {\em Michigan Math. J.}, 42(1):103--107, 1995.

\bibitem[BP89]{BP89}
Martin~T. Barlow and Edwin~A. Perkins.
\newblock Symmetric {M}arkov chains in {$\mathbb{Z}^d$}: how fast can they
  move?
\newblock {\em Probab. Theory Related Fields}, 82(1):95--108, 1989.

\bibitem[BS01]{ben-schramm}
Itai Benjamini and Oded Schramm.
\newblock Recurrence of distributional limits of finite planar graphs.
\newblock {\em Electron. J. Probab.}, 6:no. 23, 13 pp. (electronic), 2001.

\bibitem[CKR01]{CKR01}
Gruia Calinescu, Howard Karloff, and Yuval Rabani.
\newblock Approximation algorithms for the 0-extension problem.
\newblock In {\em Proceedings of the 12th Annual ACM-SIAM Symposium on Discrete
  Algorithms}, pages 8--16, Philadelphia, PA, 2001.

\bibitem[DLP13]{DLP13}
Jian Ding, James~R. Lee, and Yuval Peres.
\newblock Markov type and threshold embeddings.
\newblock {\em Geom. Funct. Anal.}, 23(4):1207--1229, 2013.

\bibitem[Kes59]{Kesten59}
Harry Kesten.
\newblock Symmetric random walks on groups.
\newblock {\em Trans. Amer. Math. Soc.}, 92:336--354, 1959.

\bibitem[Kes86]{Kesten86}
Harry Kesten.
\newblock Subdiffusive behavior of random walk on a random cluster.
\newblock {\em Ann. Inst. H. Poincar\'e Probab. Statist.}, 22(4):425--487,
  1986.

\bibitem[KLMN05]{KLMN05}
R.~Krauthgamer, J.~R. Lee, M.~Mendel, and A.~Naor.
\newblock Measured descent: {A} new embedding method for finite metrics.
\newblock {\em Geom. Funct. Anal.}, 15(4):839--858, 2005.

\bibitem[NPSS06]{NPSS06}
Assaf Naor, Yuval Peres, Oded Schramm, and Scott Sheffield.
\newblock Markov chains in smooth {B}anach spaces and {G}romov-hyperbolic
  metric spaces.
\newblock {\em Duke Math. J.}, 134(1):165--197, 2006.

\bibitem[Pin94]{azuma}
Iosif Pinelis.
\newblock Optimum bounds for the distributions of martingales in {B}anach
  spaces.
\newblock {\em Ann. Probab.}, 22(4):1679--1706, 1994.

\bibitem[Woe00]{Woess00}
Wolfgang Woess.
\newblock {\em Random walks on infinite graphs and groups}, volume 138 of {\em
  Cambridge Tracts in Mathematics}.
\newblock Cambridge University Press, Cambridge, 2000.

\end{thebibliography}


\begin{thebibliography}{1}

\bibitem{benjamini}
I.~Benjamini, H.~Duminil-Copin, G.~Kozma, A.~Yadin.
\newblock Disorder, entropy and harmonic functions.
\newblock {\em The Annals of Probability}, 43(5):2332--2373, 2015.

\bibitem{azuma}
I.~Pinelis.
\newblock Optimum bounds for the distributions of martingales in banach spaces.
\newblock {\em The Annals of Probability}, pages 1679--1706, 1994.

\end{thebibliography}

\end{document}